\documentclass[reqno, a4paper, final]{amsart}

\usepackage[T1]{fontenc}
\usepackage{lmodern}
\usepackage{graphicx}
\usepackage{subcaption}
\usepackage[final]{microtype}
\usepackage{amssymb}
\usepackage{amsthm, thmtools}
\usepackage{mathtools}
\usepackage[colorlinks, allcolors = blue]{hyperref}
\usepackage[noabbrev, capitalize, nameinlink]{cleveref}
\usepackage{textcomp}
\usepackage{xspace}

\usepackage[backend = biber,
            style = alphabetic,
            maxbibnames = 5,
            doi = false,
            isbn = false,
            url = false,
            firstinits = true]{biblatex}
\DeclareSymbolFont{extraup}{U}{zavm}{m}{n}
\DeclareMathSymbol{\vardiamond}{\mathalpha}{extraup}{87}

\numberwithin{equation}{section}
\declaretheorem[numberwithin = section]{theorem}
\declaretheorem[sibling = theorem]{lemma}
\declaretheorem[sibling = theorem]{proposition}
\declaretheorem[sibling = theorem,
                style = definition,
                qed = $\vardiamond$]{example}
\declaretheorem[sibling = theorem, style = remark, qed = $\spadesuit$]{remark}
\declaretheorem[style = remark, numbered = no]{acknowledgements}

\newcommand{\locus}[1]{rank-\(#1\)-locus}
\newcommand{\loci}[1]{rank-\(#1\)-loci}
\newcommand{\point}[1]{rank-\(#1\)-point}
\newcommand{\points}[1]{rank-\(#1\)-points}
\newcommand{\node}[1]{rank-\(#1\)-node}
\newcommand{\nodes}[1]{rank-\(#1\)-nodes}
\newcommand{\quadric}[1]{rank-\(#1\)-quadric}
\newcommand{\quadrics}[1]{rank-\(#1\)-quadrics}
\newcommand{\singularity}[1]{rank-\(#1\)-singularity}
\newcommand{\C}{\mathbb{C}}
\newcommand{\PP}{\mathbb{P}}
\newcommand{\GG}{\mathbb{G}}
\newcommand{\V}{\mathcal{V}\mkern-1mu}
\newcommand{\ybf}{\mathbf{y}}
\newcommand{\df}{\mathfrak{d}}
\newcommand{\dash}{\textthreequartersemdash\xspace}

\DeclareMathOperator{\Sing}{Sing}
\DeclareMathOperator{\rank}{rank}
\DeclareMathOperator{\corank}{corank}

\makeatletter
    \renewcommand{\@secnumfont}{\bfseries}
    \renewcommand\section{\@startsection{section}{1}
        \z@{.7\linespacing\@plus\linespacing}{0.5\linespacing}
        {\normalfont\scshape\bfseries\centering}}
\makeatother

\title{Rational Quartic Symmetroids}
\author{Martin Hels{\o}}
\address{Martin Hels{\o}, University of Oslo, Postboks 1053, Blindern, 0316 Oslo, Norway}
\email{\href{mailto:martibhe@math.uio.no}{martibhe@math.uio.no}}
\subjclass[2010]{14M12,14J26}
\keywords{Rational surfaces, determinantal varieties, linear systems of quadrics}

\begin{document}

\begin{abstract}
    We classify rational, irreducible quartic symmetroids in projective 3-space. They are either singular along a line or a smooth conic section, or they have a triple point or a tacnode.
\end{abstract}

\maketitle

\section{Introduction}

\noindent A \emph{symmetroid} $S \subset \C\PP^n$ is a hypersurface $\V(F)$, whose defining polynomial~$F$ can be written as the determinant of a symmetric matrix of linear forms.  That is, $F = \det(A)$, where
\begin{align}
    \label{eq:matrix-representation}
    A \coloneqq A(x) \coloneqq A_0x_0 + \cdots + A_nx_n,
\end{align}
and the $A_i$ are symmetric $(d \times d)$-matrices with entries in $\C$. The degree of $S$ is then $d$. In this paper, we characterise families of rational symmetroids of degree $4$ in $\C\PP^3\mkern-4mu$.

Since $S$ is identified with the symmetrix matrix $A$, we are able to talk about the rank of $S$ at a point. By evaluating, every point $x \in \PP^n$ is associated the symmetric matrix~$A(x)$. The \emph{rank} and \emph{corank} of $x$ are defined as $\rank A(x)$ and $\corank A(x)$, respectively. The symmetroid $S$ consists of the points with corank at least $1$. The \emph{\locus{k}} of~$S$ is the set of points with rank less than or equal to $k$. This is precisely the zero locus of the $(k + 1) \times (k + 1)$-minors of $A$. A \point{(d - 2)} is singular on $S$ and generically it is a node. The \locus{(d - 2)} is not necessarily equal to the singular locus.

A generic quartic symmetroid in $\PP^3$ has ten \points{2}, which are nodes, and no further singularities. If a quartic surface has a finite set of nodes \dash or more generally, rational double points \dash it is birationally equivalent to a K3-surface, and is therefore irrational. Hence, a rational quartic symmetroid has either infinitely many rational double points or a more complicated singularity. In fact, the rational quartic surfaces are either double along a curve, or have a triple point or an elliptic double point \cites[Article~96]{Jes16}{Noe89}.

The generic case was first studied by Cayley in \cite{Cay69}. An account of of this work can also be found in \cite[Chapter~IX]{Jes16}. For a real, generic quartic symmetroid with a non-empty spectrahedron, the possible arrangements of the ten nodes are described by Degtyarev and Itenberg in \cite{DI11} and by Ottem et al.\ in \cite{Ott+14}. However, the rational quartic symmetroids have not been studied systematically before. We prove the following:

\newcommand{\textmain}
{
    A rational, irreducible quartic symmetroid in $\PP^3$ is either singular along a line \dash consisting either of \points{3} or \points{2} \dash or a smooth conic section, or it has a triple point or a tacnode.
    
    For each of the families of symmetroids having precisely one of these singular loci, a general member~$S$ satisfies the following:
    \begin{enumerate}
        \item
        $S$ is singular along a line of \points{3} and has four isolated nodes,

        \item
        $S$ is singular along a line of \points{2} and has six isolated nodes,

        \item
        $S$ is singular along a smooth conic section and has four isolated nodes,

        \item
        $S$ has a triple point and six isolated nodes,

        \item
        $S$ has a tacnode and six isolated nodes.
    \end{enumerate}
    All of the isolated nodes are \points{2}. For the cases $(1)\mkern-2mu$--$(5)$, the families have dimensions $21$, $19$, $17$, $21$ and $20$, respectively.
}

\begin{theorem}
    \label{thm:main}
    \textmain
\end{theorem}

\begin{remark}
    \cref{thm:main} is not an exhaustive list of rational quartic symmetroids. Such a symmetroid can be singular along a curve with more than one component, or some of the isolated \points{2} may coincide and form more complicated singularities, or the symmetroid may have additional isolated \nodes{3}.

    A few examples are given in \cref{sec:two-lines,sec:examples}. Among them are notably the Steiner surface and Pl{\"u}cker's surface.
\end{remark}

\section{Linear Systems of Quadrics}

\noindent
Having identified the symmetroid $S$ with the symmetric matrix $A$, we have the notions of the associated quadratic form and quadric of $S$ at a point.

If $S$ has degree~$d$, let $\ybf \coloneqq [y_0, \ldots, y_{d-1}]$. The point $x \in \PP^n$ is associated the quadratic form $q(x) \coloneqq \ybf^T\mkern-7mu A(x) \ybf$. Then $Q(x) \coloneqq \V(q(x)) \subset \PP^{d-1}$ is the \emph{associated quadric} at $x$. The quadrics in the set
\begin{align*}
    W(S) \coloneqq \big\{ Q(x) \mid x \in \PP^n \big\}
\end{align*}
form the \emph{associated linear system of quadrics} of $S$. The point in $W(S)$ corresponding to the quadric $Q \subset \PP^{d-1}$ is denoted by $[Q]$. A symmetroid defined as the set of singular quadrics in a space $W$ of quadrics, is often called the \emph{discriminant hypersurface} of the space. The discriminant $D$ is given by $\det(A) = 0$, where $A$ is the matrix parametrising $W\mkern-4mu$. We have that $W = W(D)$.

\begin{remark}
    The choice of representation~\eqref{eq:matrix-representation} does not appear in the notation~$W(S)$. This is abuse of notation, because the associated system is in general not unique. The uniqueness holds in special cases. In particular, this is true for $10$-nodal quartic symmetroids in $\PP^3$ \cite[Proposition~11]{Ble+12}.

    The greatest discrepancy among the associated systems occurs for cones. Indeed, suppose that $S$ is defined by a matrix $A$ involving a variable $x_i$ that does not appear in $\det(A)$. Then $S$ is also given by the matrix $A'$  defined as $A$ with $x_i = 0$. The linear system of quadrics induced by $A'\mkern-3mu$, has lower dimension than the system induced by $A$.

    Different matrix representations can also give rise to different systems that have the same dimension. For instance, \cref{ex:2-representations} shows a symmetroid with two representations, where the \locus{2} of the symmetroid differs for the two representations.

    Our abuse of notation is justified by the fact that we only use properties that hold for all associated linear systems of quadrics that have the same \loci{k}.
\end{remark}

\noindent
Consider $\PP^9$ as the space of all quadrics in $\PP^3\mkern-4mu$. By the above, quartic symmetroids in $\PP^3$ correspond to linear subspaces $W \subset \PP^9$ with $\dim(W) \leqslant 3$. Recall that the dimension of the Grassmannian $\GG(k, n)$ of $k$-dimensional, linear subspaces of $\PP^n$ is given by
\begin{align}
    \label{eq:dim-Grassmannian}
    \dim \GG(k, n) = (k + 1)(n-k).
\end{align}
Since $\dim \GG(3, 9) = 24$, the symmetroids form a $24$-dimensional variety in the $\PP^{34}$ of all quartic surfaces. 

Let $D$ be the discriminant of the $\PP^9$ of all quadrics. The \locus{2}~$X_2$ of~$D$ is a sixfold of degree~$10$, and the \locus{1}~$X_1$ is a threefold of degree~$8$. B{\'e}zout's theorem implies that a generic linear $3$-space in $\PP^9$ intersects $X_2$ in ten points and avoids $X_1$. A generic quartic symmetroid has therefore ten \points{2}, which are nodes, and no \points{1}. Moreover, $\Sing(D) = X_2$, so it has no other singularities.

Next, we have collected a few results about spaces of quadrics. For a proof of the first lemma, see \cite[Lemma~1.13]{Il+16}.

\begin{lemma}
    \label{lem:singular-at-base-point}
    Let $\PP^n$ be a linear space of quadrics in a projective space $\PP^d$ and let $B \subset \PP^d$ be the base locus of the quadrics in $\PP^n\mkern-4mu$. Let $D \subset \PP^n$ be the discriminant. If $[Q] \in \PP^n$ is a point such that $Q$ is a singular quadric with a singularity at $p \in B$, then the discriminant $D$ is singular at $[Q]$.
\end{lemma}

\noindent It is not true in general that if $[Q] \in \Sing(D)$, then $\Sing(Q) \cap B \neq \varnothing$. Below is a counterexample to the converse of \cref{lem:singular-at-base-point}:

\begin{example}
    \label{ex:counterexample-converse}
    Consider the symmetroid $S \subset \PP^3$ defined by the determinant of the matrix
    \begin{align*}
        \begin{bmatrix}
            0 & 
            x_0 + x_1 - x_2 - x_3 &
            -2x_1 - x_2 + x_3 &
            x_3 
            \\
            x_0 + x_1 - x_2 - x_3 &
            0 &
            x_0 - x_1 - 2x_2 &
            x_0 
            \\
            -2x_1 - x_2 + x_3 &
            x_0 - x_1 - 2x_2 &
            0 &
            x_1 
            \\
            x_3 &
            x_0 &
            x_1 &
            x_2
        \end{bmatrix}
        \mkern-7mu.
    \end{align*}
    It has rank $2$ along a line and in six additional points.
    The associated quadrics $Q_i$ at the points 
    \begin{align*}
        [Q_1] \coloneqq [1 : 0 : 0 : 0],
        \
        [Q_2] \coloneqq [0 : 1 : 0 : 0], 
        \
        [Q_3] \coloneqq [0 : 0 : 1 : 0],
        \
        [Q_4] \coloneqq [0 : 0 : 0 : 1]
    \end{align*}
    are
    \begin{equation*}
        \begin{aligned}
            Q_1 &= \V(y_1(y_0 + y_2 + y_3)), 
            \\[0.3ex]
            Q_2 &= \V(y_0y_1 - 2y_0y_2 - y_1y_2 + y_2y_3),
        \end{aligned}
        \qquad
        \begin{aligned}
            Q_3 &= 
            \V\big(2y_0y_1 + 2y_0y_2 + 4y_1y_2 - y_3^2\big), 
            \\[0.3ex] 
            Q_4 &= \V(y_0(y_1 - y_2 - y_3)).
        \end{aligned}    
    \end{equation*}
    These four span the associated web of quadrics. The intersection $Q_1 \cap Q_2 \cap Q_3 \cap Q_4$ consists of the four coplanar points
    \begin{align*}
        p_1 \coloneqq [1 : 0 : 0 : 0],
        \quad
        p_2 \coloneqq [0 : 1 : 0 : 0], 
        \quad
        p_3 \coloneqq [0 : 0 : 1 : 0],
        \quad
        p_4 \coloneqq [-1 : 1 : 1 : 0].
    \end{align*}    
    Note that $[Q_1] \in \Sing(S)$ is \point{2}. The quadric $Q_1$ is singular along the line $L \coloneqq \V(y_1, y_0 + y_2 + y_3)$. None of the base points $p_i$ lie on $L$.
\end{example}

\noindent The key to \cref{ex:counterexample-converse} is that the node $[Q_1]$ has low rank. We prove a strengthened version of \cref{lem:singular-at-base-point} for nodes with corank $1$:

{
    \pretolerance = 228
    \begin{lemma}[{\cite[Lemma~1.1]{Wal81}}]
        \label{lem:singularity-is-a-base-point}
        Let $\PP^n$ be a linear space of quadrics in a projective space $\PP^d$ and let $B \subset \PP^d$ be the base locus of the quadrics in $\PP^n\mkern-4mu$. Let $D \subset \PP^n$ be the discriminant. Then $D$ has degree~$d+1$. If $[Q] \in D$ is a \point{d}, let $p$ be the singular point of $Q$. Then $p \in B$ if and only if $D$ is singular at $[Q]$.
    \end{lemma}
}

\begin{proof}
    By \cref{lem:singular-at-base-point}, the `only if' direction is true regardless of $\corank [Q]$. For the `if' direction, choose coordinates such that $[Q] \coloneqq [1 : 0 : \ldots : 0] \in \PP^n\mkern-4mu$. It is always possible to conjugate $A(x)$ in such a way that one of the matrices $A_i$ in \eqref{eq:matrix-representation} is diagonal. We may therefore assume that the matrix defining $Q$ is
    \begin{equation*}
        \setlength{\arraycolsep}{1pt}
        A_0
        \coloneqq
        \begin{bmatrix}
            a_0 & &
            \multicolumn{2}{c}
            {
                \text{\kern1.8em\smash{\raisebox{-2ex}{\huge 0}}}
            } 
            \\
            & \ddots & &
            \\
            & & a_{d-1} &
            \\
            \multicolumn{2}{c}
            {
                \text{\kern-1.3em\smash{\raisebox{0.2ex}{\huge 0}}}
            }
            & & 0
        \end{bmatrix}
        \mkern-7mu.
    \end{equation*}
    The $a_i$ are non-zero since the rank of $[Q]$ is $d$. Hence $Q = \V\big(a_0x^2_0 + \cdots + a_{d-2}x^2_{d-1}\big)$, so $p = [0 : \ldots : 0 : 1] \in \PP^d\mkern-4mu$.
    
    Denote the entries in $A_1x_1 + \cdots + A_{n}x_{n}$ by $l_{ij} \in \C[x_1,\ldots,x_n]_1$, for $0 \leqslant i \leqslant j \leqslant d$. The discriminant~$D$ is given by the determinant of $A(x) \coloneqq A_0x_0 + [l_{ij}]$. Since $[Q]$ is a singular point on $D$, then $\det(A(x))$ cannot contain any $x_0^d x_i$ terms for $i = 0, \ldots, n$. It follows that $l_{dd} = 0$. This implies that $p$ is in the base locus $B$.
\end{proof}

\noindent We will need some special properties of $1$-dimensional linear systems of quadrics:

\begin{lemma}
    \label{lem:rank2-pencil}
    Let $P$ be a pencil of \quadrics{2} in $\PP^n\mkern-4mu$, with $n \geqslant 2$. The base locus of~$P$ consists of a hyperplane $H$ and a linear subspace $L \not\subset H$ of codimension~$2$.
\end{lemma}

\begin{proof}
    The singular locus of a \quadric{k} is a linear space with codimension~$k$. A \quadric{2} is therefore the union of two hyperplanes. Let $Q_1 \coloneqq H_1 \cup H'_1$ and $Q_2 \coloneqq H_2 \cup H'_2$ be two generators for $P\mkern-3mu$. If there is no relation between the hyperplanes $H_1$ and $H'_1$ and the hyperplanes $H_2$ and $H'_2$, then the intersection $Q_1 \cap Q_2$ consists of four general linear, $(n-2)$-dimensional varieties. Thus there exist quadrics with rank greater than $2$ passing through $Q_1 \cap Q_2$. This contradicts the fact that $P$ consists of \quadrics{2}.

    Hence, there are two possibilities: Either $H_2$ and $H'_2$ are linear combinations of $H_1$ and $H'_1$, or $Q_1$ and $Q_2$ have a hyperplane $H$ in common. If $H_2$ and $H'_2$ are linear combinations of $H_1$ and $H'_1$, then $P$ contains two \quadrics{1}, unless $Q_1$ and $Q_2$ have a hyperplane in common. Suppose that $H_1 = H_2 \eqqcolon H\mkern-2mu$. Then the base locus of $P$ consists of $H$ and $L \coloneqq H'_1 \cap H'_2$. If $L \subset H\mkern-2mu$, then $P$ contains the double hyperplane $H\mkern-2mu$, which is a \quadric{1}. The claim follows by elimination.
\end{proof}

\begin{lemma}
    \label{lem:rank3-pencil}
    Let $P$ be a pencil of quadrics in $\PP^3\mkern-4mu$. Assume that a general quadric in $P$ has rank $3$. Let $Q_1, Q_2 \in P$ be two \quadrics{3}. Then there are no smooth quadrics in $P$ and one of the following is true:
    \begin{enumerate}
        \setlength{\itemsep}{1ex}

        \item
        The quadrics in $P$ have a common singular point $\Sing(Q_1) = \Sing(Q_2)$, and the \locus{2} of $P$ is a scheme of length $3$.

        \item
        The quadrics in $P$ have a common tangent plane along the line $L$ spanned by $\Sing(Q_1)$ and $\Sing(Q_2)$. In this case, $P$ contains a single \quadric{2}.
    \end{enumerate}
\end{lemma}

\begin{proof}
    The discriminant of a pencil of quadrics in $\PP^3$ is either a scheme of length~$4$ or the whole line. Since a general quadric in $P$ is of rank $3$, the discriminant~$D$ equals $P\mkern-3mu$. Hence there are no smooth quadrics in $P\mkern-3mu$. Moreover, since $D$ is given by the zero polynomial, it is singular at all points. By \cref{lem:singularity-is-a-base-point}, the singular point of each \quadric{3} in $P$ is a base point. It follows that the quadrics either have a common singular point $(1)$ or the singular points form a line $(2)$.

    \begin{enumerate}
        \setlength{\itemsep}{1ex}

        \item
        The intersection $Q_1 \cap Q_2$ consists of four concurrent lines, $L_1$, $L_2$, $L_3$ and $L_4$. Let $H_{ij}$ be the plane spanned by $L_i$ and $L_j$, and let $H_{kl}$ be the plane spanned by the remaining two lines, $L_k$ and $L_l$. The union $Q_{ij} \coloneqq H_{ij} \cup H_{kl}$ is a \quadric{2} contained in $P\mkern-3mu$. There are three such \quadrics{2}, $Q_{12}$, $Q_{13}$ and $Q_{14}$. If some of the $L_i$ coincide, then some of the $Q_{ij}$ coincide and in some cases become \quadrics{1}, but the length of the \locus{2} remains $3$.

        \item
        Both $Q_1$ and $Q_2$ contain the line $L$, and they have apexes along this line. Thus they share a tangent plane along $L$. The intersection $Q_1 \cap Q_2$ consists of $L$, counted with multiplicity~$2$, and a conic section $C$. The union of the tangent plane along $L$ and the plane spanned by $C$, is the only \quadric{2} in $P\mkern-3mu$. \qedhere
    \end{enumerate} 
\end{proof}

\noindent
The following simple observation is useful for excluding possible symmetroids:

\begin{lemma}
    \label{lem:curve-in-base-locus}
    Let $S \subset \PP^3$ be a quartic symmetroid and assume that the base locus of $W(S)$ contains a curve $C$. Then $S$ is reducible.
\end{lemma}

\begin{proof}
    The base locus of $W(S)$ is an intersection of quadrics in $\PP^3\mkern-4mu$. Hence, $C$ has either a line, a smooth conic section, a twisted cubic curve or an irreducible quartic curve as a component. Suppose that $C$ contains a line~$L$. 
    Choose coordinates such that $L \coloneqq \V(x_2, x_3)$. The quadrics containing $L$ are parametrised by the matrix
    \begin{align*}
        \setlength{\arraycolsep}{6pt}
        A_1
        \coloneqq
        \begin{bmatrix}
                0  &      0 & x_{02} & x_{03} \\
                0  &      0 & x_{12} & x_{13} \\
            x_{02} & x_{12} & x_{22} & x_{23} \\
            x_{03} & x_{13} & x_{23} & x_{33}
        \end{bmatrix}
        \mkern-7mu.
    \end{align*}
    The determinant of $A_1$ is $(x_{02}x_{13} - x_{03}x_{12})^2\mkern-4mu$, so $S$ is a double quadric.
    
    Suppose that $C$ contains a smooth conic section~$K$. Choose coordinates such that $K$ is the intersection of the plane $\V(x_0)$ and the quadric $\V \big( x_1^2 + x_2^2 + x_3^2 \big)\mkern-3mu$. The space of quadrics passing through $K$ is then parametrised by 
    \begin{align*}
        \setlength{\arraycolsep}{6pt}
        A_2
        \coloneqq
        \begin{bmatrix}
            x_{00} & x_{01} & x_{02} & x_{03} \\
            x_{01} & x_{22} &   0    &   0    \\
            x_{02} &   0    & x_{22} &   0    \\
            x_{03} &   0    &   0    & x_{22}
        \end{bmatrix}
        \mkern-7mu.
    \end{align*}
    Since $\det(A_2) = \big( x_{00} x_{22} - x_{01}^2 - x_{02}^2 - x_{03}^2 \big) x_{22}^2$, it follows that $S$ is the union of a quadric and a double plane.

    The quadrics that contain a twisted cubic curve~$T$ only form a net. Thus, if $C$ contains $T\mkern-4mu$, then $S$ is a cone. Moreover, choose coordinates such that $T$ is given as the intersection of $\V \big( x_0x_2 - x_1^2 \big)$, $\V \big( x_0x_3 - x_1x_2 \big)$ and $\V \big( x_1x_3 - x_2^2 \big)\mkern-3mu$. Then $S$ is defined by
    \begin{align*}
        \setlength{\arraycolsep}{6pt}
        A_3
        \coloneqq
        \begin{bmatrix}
              0    &      \phantom{-}0     &   \phantom{-2}x_{02}  & x_{03} \\
              0    &             -2x_{02}  &  \phantom{2}{-x_{03}} & x_{13} \\
            x_{02} &  \phantom{2}{-x_{03}} &             -2x_{13}  &   0    \\
            x_{03} &   \phantom{-2}x_{13}  &      \phantom{-}0     &   0
        \end{bmatrix}
        \mkern-7mu.
    \end{align*}
    The determinant of $A_3$ is $\big( x_{02}x_{13} - x_{03}^2 \big)^{\mkern-4mu2}\mkern-4mu$, so $S$ is a double quadratic cone.
    
    Finally, a quartic curve~$Q$ is the intersection of two quadrics. If $C$ contains~$Q$, then the equation for $S$ is defined by only two variables. It follows that $S$ is the union of four planes.
\end{proof}

\noindent
We isolate the result from the first part of the proof of \cref{lem:curve-in-base-locus} for easy reference:

\begin{lemma}
    \label{lem:square-discriminant}
    Let $P$ be a linear space of quadrics and suppose that the base locus of~$P$ contains a line. Then the discriminant of $P$ is a square.
\end{lemma}

\section{Quartic Symmetroids with a Double Line}

\noindent
Let $p_1, \ldots, p_9 \in \PP^2$ be nine points that are not the complete intersection of two cubic curves. Consider the linear system $\df$ of quartic curves passing twice through $p_1$ and once through each of the points $p_2, \ldots, p_9$. Let $\varphi \colon \PP^2 \dashrightarrow \PP^3$ be the map induced by $\df$. The image $\varphi \big( \PP^2 \big) \subset \PP^3$ is a quartic surface with a double line. Any quartic surface $S \subset \PP^3$ with a double line arises this way \cite[Article~79]{Jes16}. Consequently, $S$ is rational.

For a quartic symmetroid $S$ with a double line $L$, there are two possibilities: The points along $L$ are either generically \points{3}, or they are all \points{2}. We show that if $S$ is a generic symmetroid with a double line containing \points{3}, then $S$ has four \points{2} outside of $L$ and no further singularities. The family of such symmetroids is $21$-dimensional. Likewise, we show that if $S$ is a generic symmetroid with a line of \points{2}, then it has six \points{2} outside of $L$. Symmetroids of this type form a $19$-dimensional family.

\begin{lemma}
    \label{lem:rank2-line-implies-base-points}
    Let $S \subset \PP^3$ be a general, irreducible, quartic symmetroid with a line~$L$ of \points{2}. Then $W(S)$ has four general, coplanar base points.
\end{lemma}

\begin{proof}
    The line $L$ corresponds to a pencil $P \subset W(S)$ of \quadrics{2}. By \cref{lem:rank2-pencil}, the base locus of $P$ consists of a plane $H$ and a line~$l \not\subset H$.

    Let $Q_1, Q_2 \in W(S)$ be such that $Q_1$, $Q_2$ and $P$ generate $W(S)$. We may assume that $[Q_1]$ and $[Q_2]$ are not in the \locus{2} of $S$, so the plane $H$ is not contained in either of the $Q_i$. By B{\'e}zout's theorem, each $Q_i$ intersects $H$ in a conic~$C_i$. Similarly, $Q_i$ intersects the line $l$ in two points, $p_i$ and $p'_i$, each. Generically, none of the points $p_1$, $p'_1$, $p_2$ and $p'_2$ coincide. However, B{\'e}zout's theorem implies that $C_1$ and $C_2$ generically intersect in four general points.
\end{proof}

\noindent
The connection between symmetroids with a line of \points{2} and webs of quadrics with four coplanar base points, is also true in the other direction:

\begin{lemma}
    \label{lem:base-points-imply-rank2-line}
    Let $W$ be a web of quadrics in $\PP^3$ with four general, coplanar base points. Generically, the discriminant $D \subset \PP^3$ of $W\mkern-4mu$, has a line of \points{2} and six additional \points{2}.
\end{lemma}

\begin{proof}
    Consider the $\PP^5$ of quadrics through the four coplanar points $p_1, p_2, p_3, p_4 \in \PP^3\mkern-4mu$. We will now describe the \points{2} in $\PP^5$ by finding the \quadrics{2} passing through $p_1$, $p_2$, $p_3$ and $p_4$. First, let $H$ be the plane spanned by $p_1$, $p_2$, $p_3$ and $p_4$. Then the union of $H$ and any plane $H' \subset \PP^3$ is a \quadric{2} containing the base points. The set $X$ of all such unions forms a $\PP^3 \subset \PP^5\mkern-4mu$.

    Next, let $H_{ij}$ be a plane containing the line $L_{ij}$ spanned by the points $p_i$ and $p_j$. Let $H_{kl}$ be a plane containing the line $L_{kl}$ spanned by the remaining two points, $p_k$ and $p_l$. Then the union of $H_{ij}$ and $H_{kl}$ is a \quadric{2} containing $p_1$, $p_2$, $p_3$ and $p_4$. The set $X_{ij}$ of all such unions forms a quadratic surface in $\PP^5\mkern-4mu$. Since the points are in general position, there are in total three such surfaces of \points{2}, namely $X_{12}$, $X_{13}$ and $X_{14}$.

    By B{\'e}zout's theorem, a generic, linear $3$-space $W \subset \PP^5$ intersects $X$ in a line and the three surfaces $X_{ij}$ in two points each. This proves the claim.
\end{proof}

\noindent
In \cref{lem:base-points-imply-rank2-line}, if we omit the assumption that the coplanar base points are general, then three of them can lie on a line~$L$. In that case, the base locus contains $L$ and \cref{lem:curve-in-base-locus} states that $D$ is reducible.

The next result is immediate from \cref{lem:rank2-line-implies-base-points,lem:base-points-imply-rank2-line}:

\begin{proposition}
    \label{prop:general-rank2-line}
    Let $S \subset \PP^3$ be a general quartic symmetroid with a line of \points{2}. Then $S$ has six additional \points{2}.
\end{proposition}

\noindent
%Note that the surface $S$ in \cref{ex:counterexample-converse} is a generic quartic symmetroid with a line of \points{2}.
%
The construction indicated by \cref{lem:rank2-line-implies-base-points,lem:base-points-imply-rank2-line} allows us to count the number of quartic symmetroids with a line of \points{2}:

\begin{proposition}
    \label{prop:dimension-rank2-line}
    The family of quartic symmetroids with a line of \points{2} is $19$-dimensional.
\end{proposition}

\begin{proof}
    \cref{lem:rank2-line-implies-base-points,lem:base-points-imply-rank2-line} imply that a generic quartic symmetroid with a line of \points{2} is obtained by choosing four coplanar points $p_1, p_2, p_3, p_4 \in \PP^3\mkern-4mu$ in general position, and then choosing a general $\PP^3$ in the $\PP^5$ of quadrics passing through $p_1$, $p_2$, $p_3$ and $p_4$.
    
    In how many ways can $p_1$, $p_2$, $p_3$ and $p_4$ be chosen? The first three points can be chosen freely, and the last point $p_4$ must lie in the plane spanned by $p_1$, $p_2$ and $p_3$. Hence four coplanar points in $\PP^3$ correspond to a point in $\PP^3 \times \PP^3 \times \PP^3 \times \PP^2\mkern-4mu$, which is an $11$-dimensional space. There are points in $\PP^3 \times \PP^3 \times \PP^3 \times \PP^2$ that do not correspond to four general, coplanar points, since the same point in $\PP^3$ is taken more than once or since the points are not in a general position. However, excluding these exceptions do not affect the dimension.
    
    By \eqref{eq:dim-Grassmannian}, the Grassmannian $\GG(3, 5)$ of linear $3$-spaces in the $\PP^5$ of quadrics through $p_1$, $p_2$, $p_3$ and $p_4$, is $8$-dimensional. In total, the family of quartic symmetroids with a line of \points{2} has dimension $11 + 8 = 19$.
\end{proof}

\noindent The symmetroids with a line of \points{2} only make up a small fraction of the quartic symmetroids with a double line:

\begin{proposition}
    \label{thm:dimension-double-line}
    The family of irreducible quartic symmetroids with a double line is $21$-dimensional.
\end{proposition}

\begin{proof}
    Let $S$ be a quartic symmetroid with a double line~$L$. The case where~$L$ consists of only \points{2} is covered by \cref{prop:dimension-rank2-line}. Assume therefore that~$L$ contains a \point{3}~$[Q]$. Let $p$ be the singular point of $Q$. \cref{lem:singularity-is-a-base-point} implies that $p$ is a base point for $W(S)$. Moreover, all the quadrics along $L$ are singular at~$p$. Otherwise, the singular points form a line, which by \cref{lem:singularity-is-a-base-point} is contained in the base locus of $W(S)$. \cref{lem:curve-in-base-locus} then states that $S$ is reducible.

    Consider the $\PP^8$ of all quadrics through $p$, and let $D$ be its discriminant. It imposes three conditions to require that a quadric passing through $p$, is singular at $p$. Hence the set $X_p \subset \PP^8$ of all quadrics in $\PP^3$ that are singular at $p$, is a linear $5$-space. By \cref{lem:singular-at-base-point}, $X_p$ is contained in $\Sing(D)$. Let $W \subset \PP^8$ be a linear $3$-space that intersects $X_p$ in a line. Then the discriminant of $W$ is a quartic symmetroid with a double line.

    Consider the subvariety $Y \subseteq \GG(k, n)$ consisting of the linear, $k$-dimensional subspaces $K$ that intersect a fixed linear, $m$-dimensional subspace $M$, such that $\dim(K \cap M) \geqslant l\mkern-2mu$. By \cite[Example~11.42]{Har92}, the dimension of $Y$ is given by the following formula:
    \begin{align}
        \label{eq:dim-fixed-intersection}
        \dim(Y) = (l + 1)(m - l) + (k - l)(n - k).
    \end{align}
    Thus the set of $3$-spaces in $\PP^8$ that meet $X_p$ in a line, is $18$-dimensional. In total, we obtain a $21$-dimensional family of quartic symmetroids with a double line, by letting the base point $p$ be arbitrary in $\PP^3\mkern-4mu$.
\end{proof}

\noindent
The construction from the proof of \cref{thm:dimension-double-line} allows us to determine the number of extra singularities:

\begin{proposition}
    \label{prop:general-rank3-line}
    Let $S \subset \PP^3$ be a general quartic symmetroid that is singular along a line~$L$ of \points{3}. Then $S$ has four additional nodes.
\end{proposition}

\begin{proof}
    We continue with the notation from the proof of \cref{thm:dimension-double-line}. Choose coordinates such that $p \coloneqq [1 : 0 : 0 : 0]$. Then $\PP^8$ of quadrics that pass through $p$ is parametrised by the matrix
    \begin{align*}
        A
        \coloneqq
        \begin{bmatrix}
               0   & x_{01} & x_{02} & x_{03} \\
            x_{01} & x_{11} & x_{12} & x_{13} \\
            x_{02} & x_{12} & x_{22} & x_{23} \\
            x_{03} & x_{13} & x_{23} & x_{33}
        \end{bmatrix}
        \mkern-7mu.
    \end{align*}
    Furthermore, we have that $X_p = \V(x_{01}, x_{02}, x_{03})$ and the \locus{2} $X_2$ of $D$ is a fivefold of degree~$10$.

    Using this explicit description, we compute that the tangent space at a general point on $X_2$ is $5$-dimensional, but the tangent space at a point on $X_2$ contained in $X_p$ is $6$-dimensional. In fact, $\Sing(X_2) = X_2 \cap X_p$ set-theoretically. Since $L \subset X_p$, it contains in general three \points{2}, $p_1$, $p_2$ and $p_3$. This can either be seen from the matrix $A$ or \cref{lem:rank3-pencil}. The web $W(S) \subset \PP^8$ intersects the three tangent spaces $T_{p_i}X_2$ in a line each. Thus the intersection multiplicity of $W(S)$ and $X_2$ is at least~$2$ at each of the points $p_1$, $p_2$ and~$p_3$. Since $W(S) \cap X_2$ has length~$10$, it follows that $W(S)$ meets $X_2$ in at most four points outside of $L$. Moreover, the intersection multiplicity of $W(S)$ and $X_2$ is generically $2$ at the points $p_1$, $p_2$ and $p_3$. This proves the claim.
\end{proof}

\begin{figure}[hbtp]
    \centering
    \begin{minipage}{0.5\textwidth-1ex}
        \centering
        \includegraphics[height = 0.28\textheight]{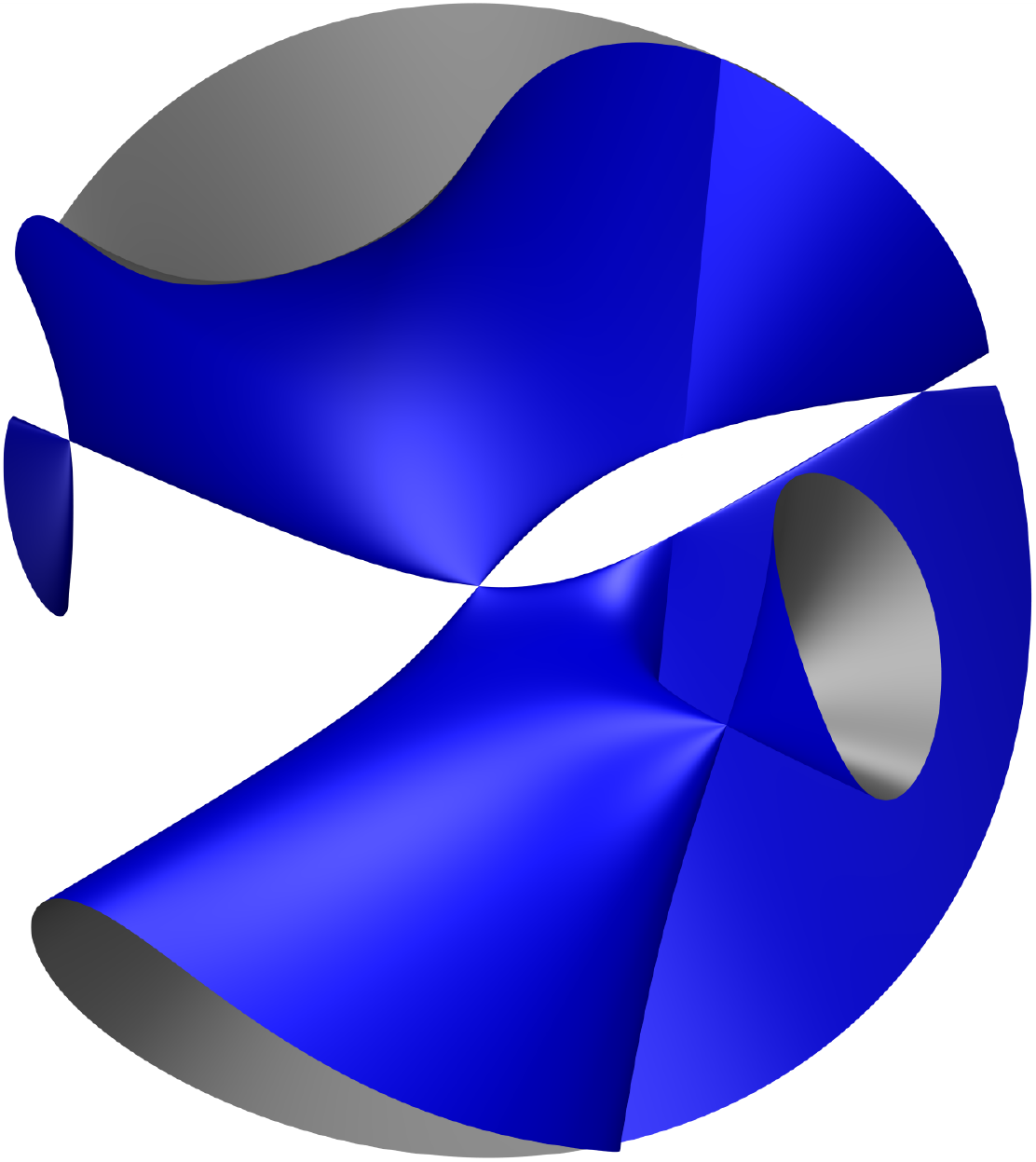}
        \subcaption{Double line of \points{3}}
        \label{fig:rank3-line}
    \end{minipage}
    \begin{minipage}{0.5\textwidth-1ex}
        \centering
        \includegraphics[height = 0.28\textheight]{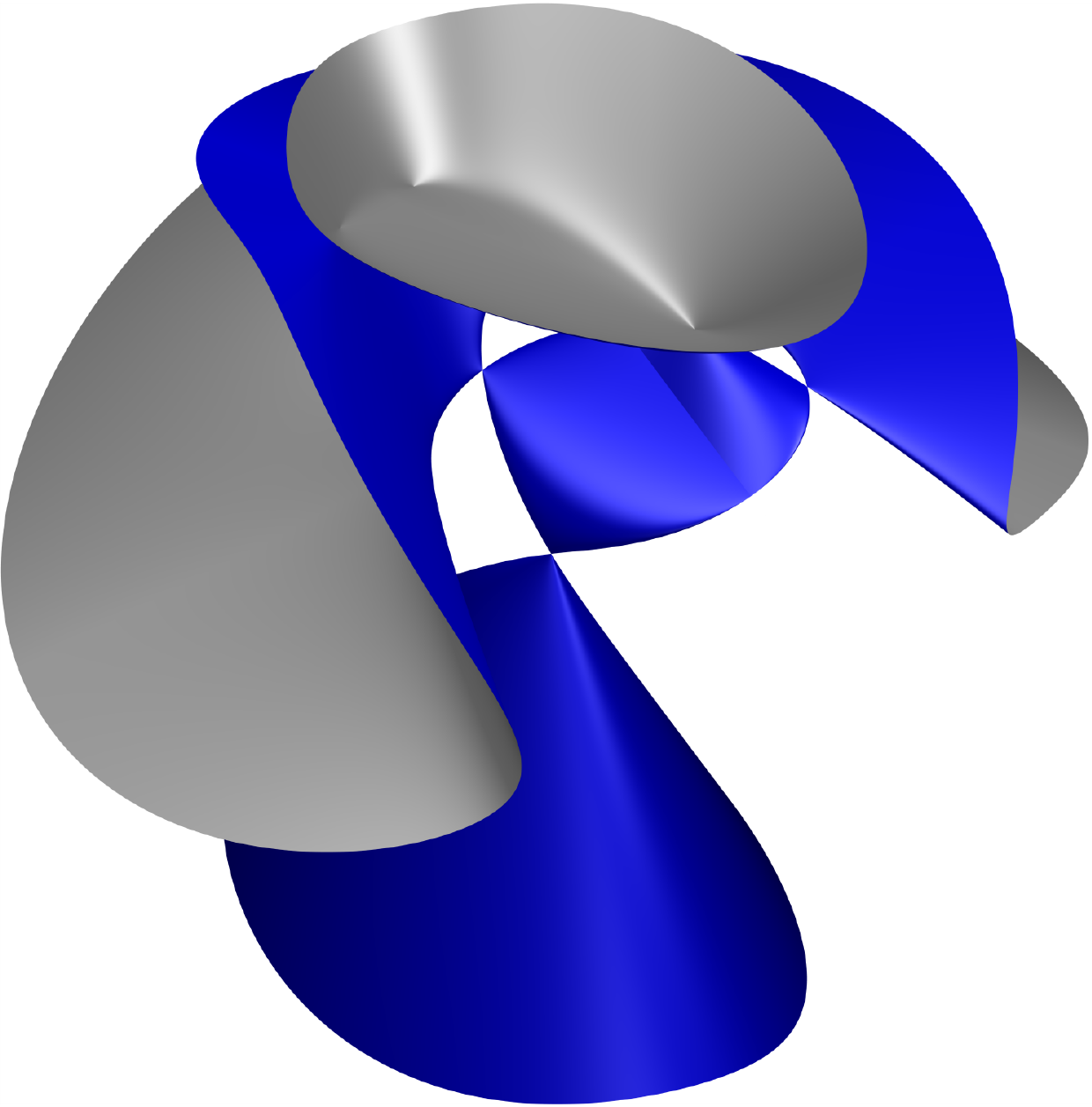}
        \subcaption{Double line of \points{2}}
        \label{fig:rank2-line}
    \end{minipage}
    \caption{General quartic symmetroids with a double line. The surface in \hyperref[fig:rank3-line]{(A)} has four real nodes, and the surface in \hyperref[fig:rank2-line]{(B)} has six real nodes.}
\end{figure}

\section{Quartic Symmetroids with a Double Conic Section}

\noindent Let $S$ be an irreducible quartic surface $S \subset \PP^3$ with a double conic. Then $S$ can be realised as the projection of a quartic del Pezzo surface $P \subset \PP^4$ \cite[Theorem~8.6.4]{Dol12}. Consequently, $S$ is rational.

We show that if $S$ is an irreducible quartic symmetroid with a double conic $C\mkern-1mu$, then there are no \points{3} on $C\mkern-1mu$. Furthermore, if $S$ is a generic symmetroid with a double conic, then it has four \points{2} outside of $C\mkern-1mu$.

\begin{proposition}
    \label{thm:double-conics-have-rank2}
    Let $S \subset \PP^3$ be an irreducible quartic symmetroid that is double along a smooth conic section~$C\mkern-1mu$. Then $C$ is contained in the \locus{2} of $S$.
\end{proposition}

\begin{proof}
    Assume for contradiction that $C$ is not contained in the \locus{2} of $S$. A generic point $[Q_1] \in C$ is then a \node{3}. By \cref{lem:singularity-is-a-base-point}, the singular point~$p$ of $Q_1$ is a base point for $W(S)$. If $[Q_2] \in C$ is another point such that $Q_2$ is singular at $p$, then all the quadrics in the pencil $\langle Q_1, Q_2 \rangle$ are singular at $p$. By \cref{lem:singular-at-base-point}, the line $L$ spanned by $[Q_1]$ and $[Q_2]$ is contained in $\Sing(S)$. Since $C$ is smooth, $L$ is not a component of $C$. Let $H$ be the plane spanned by $C$. The intersection of $H$ and $S$ contains at least $L$ and two times $C\mkern-1mu$, so $H$ must be a component in $S$. This contradicts the irreducibility of $S$. In conclusion, the different \nodes{3} on $C$ give rise to different base points of $S$. Hence, the base locus of $W(S)$ contains a curve. \cref{lem:curve-in-base-locus} implies that $S$ is reducible, which is impossible.
\end{proof}

\begin{remark}
    \cref{thm:double-conics-have-rank2} is not true for singular conic sections, as shown by \cref{ex:2-representations}.
\end{remark}

\noindent We have an analogue to \cref{lem:rank2-line-implies-base-points} for quartic symmetroids with a double conic:

\begin{proposition}
    Let $S \subset \PP^3$ be a general quartic symmetroid with a smooth conic section $C$ of \points{2}. Then $W(S)$ has four general base points.
\end{proposition}

\begin{proof}
    Let $Q \coloneqq H_1 \cup H_2$ and $Q' \coloneqq H'_1 \cup H'_2$ be quadrics corresponding to two points on $C\mkern-1mu$, where $H_1$, $H_2$, $H'_1$ and $H'_2$ are planes. Generically, $Q$ and $Q'$ intersect in four lines, $L_1, L'_1 \subset H_1$ and $L_2, L'_2 \subset H_2$. These lines constitute the base locus of the pencil $\langle Q, Q' \rangle$ generated by $Q$ and $Q'\mkern-5mu$.
    
    Letting $Q'$ run through all of the points in $C\mkern-1mu$, we obtain a pencil $P_i$ of line pairs $L_i \cup L'_i$ in both planes $H_i$. Since there are no smooth conic sections in $P_i$, it follows from \cref{lem:rank2-pencil} that the base locus of $P_i$ must contain a line, say $L_i$. Let $N$ be the net of quadrics corresponding to the plane spanned by $C$. Then $L_1$ and $L_2$ are contained in the base locus of $N\mkern-4mu$. The web $W(S)$ is generated by $N$ and a quadric $K \not\in N\mkern-4mu$. Generically, $K$ intersects $L_1$ and $L_2$ in two points each, so $W(S)$ has four general base points.
\end{proof}

\begin{remark}
    \label{rmk:no-analogue}
    There is no analogue to \cref{lem:base-points-imply-rank2-line} for conic sections. Let $W$ be a web of quadrics with four general base points. Generically, the discriminant of $W$ does \emph{not} contain a conic section of \points{2}.
    
    Indeed, consider the $\PP^5$ of quadrics through four general points $p_1, p_2, p_3, p_4 \in \PP^3$ and let $D$ be its discriminant. We shall describe the \locus{2} of $D$. Let $H \subset \PP^3$ be the plane spanned by three of the points, $p_i$, $p_j$ and $p_k$, and let $H_l \subset \PP^3$ be a plane containing the remaining point $p_l$. The union of $H$ and $H_l$ is a \quadric{2} passing through the four base points. The set $X_l$ of all such unions forms a plane in $\PP^5\mkern-4mu$. Hence there are four planes, $X_1$, $X_2$, $X_3$ and $X_4$, in the \locus{2} of $D$. In addition, the \locus{2} of $D$ contains the three quadratic surfaces $X_{12}$, $X_{13}$ and $X_{14}$, as described in the proof of \cref{lem:base-points-imply-rank2-line}. In total, the \locus{2} of $D$ is a surface of degree $10$. By B{\'e}zout's theorem, a generic linear $3$-space $W \subset \PP^5$ contains $10$ \points{2}. Hence, $W$ must be in a special position in order to contain a conic section of \points{2}.
\end{remark}

\noindent
We can still deduce the number of additional \points{2} for a general quartic symmetroid with a double conic:

\begin{proposition}
    \label{prop:general-conic}
    Let $S \subset \PP^3$ be a general quartic symmetroid that is singular along a smooth conic section. Then $S$ has four additional nodes.
\end{proposition}

\begin{proof}
    We continue with the notation from \cref{rmk:no-analogue}.
    The union of $H$ and a plane containing the line spanned by $p_l$ and $p_k$, is a quadric that lie in $X_l \cap X_{ij}$. The intersection $X_l \cap X_{ij}$ is the line $L_l$ of all such quadrics. Suppose that $W(S) \subset \PP^5$ intersects $X_{12}$ in a conic section $C\mkern-1mu$. Generically, $W(S)$ intersects the quadratic surfaces $X_{13}$ and $X_{14}$ in two points each, and the planes $X_l$ in a point each. Since the lines $L_l$ meet $C$ in a point, $W(S)$ does not intersect $X_l$ outside of $C\mkern-1mu$. Hence, $W(S)$ has generically four isolated \points{2}.
\end{proof}

\noindent
We can count the number of symmetroids with a double conic:

\begin{proposition}
    \label{prop:dimension-conic}
    The family of quartic symmetroids with a double smooth conic section is $17$-dimensional.
\end{proposition}

\begin{proof}
    As in the proof of \cref{prop:general-conic}, a quartic symmetroid with a double smooth conic section corresponds to a linear $3$-space $W \subset \PP^5$ that intersects $X_{12}$ in a conic section. This is the same as saying that $W$ intersects the $\PP^3$ spanned by $X_{12}$, in a plane. It follows from \eqref{eq:dim-fixed-intersection} that the family of $3$-spaces that intersect $X_{12}$ in a conic is $5$-dimensional.

     The calculation above shows the number of linear systems, with a conic section of \points{2}, having fixed base points $p_1$, $p_2$, $p_3$ and $p_4$. A choice of base points corresponds to a point in $\PP^3 \times \PP^3 \times \PP^3 \times \PP^3$, which is $12$-dimensional. In total, the family of quartic symmetroids with a double conic has dimension $5 + 12 = 17$.
\end{proof}

\begin{figure}[htbp]
    \centering
    \includegraphics[height = 0.28\textheight]{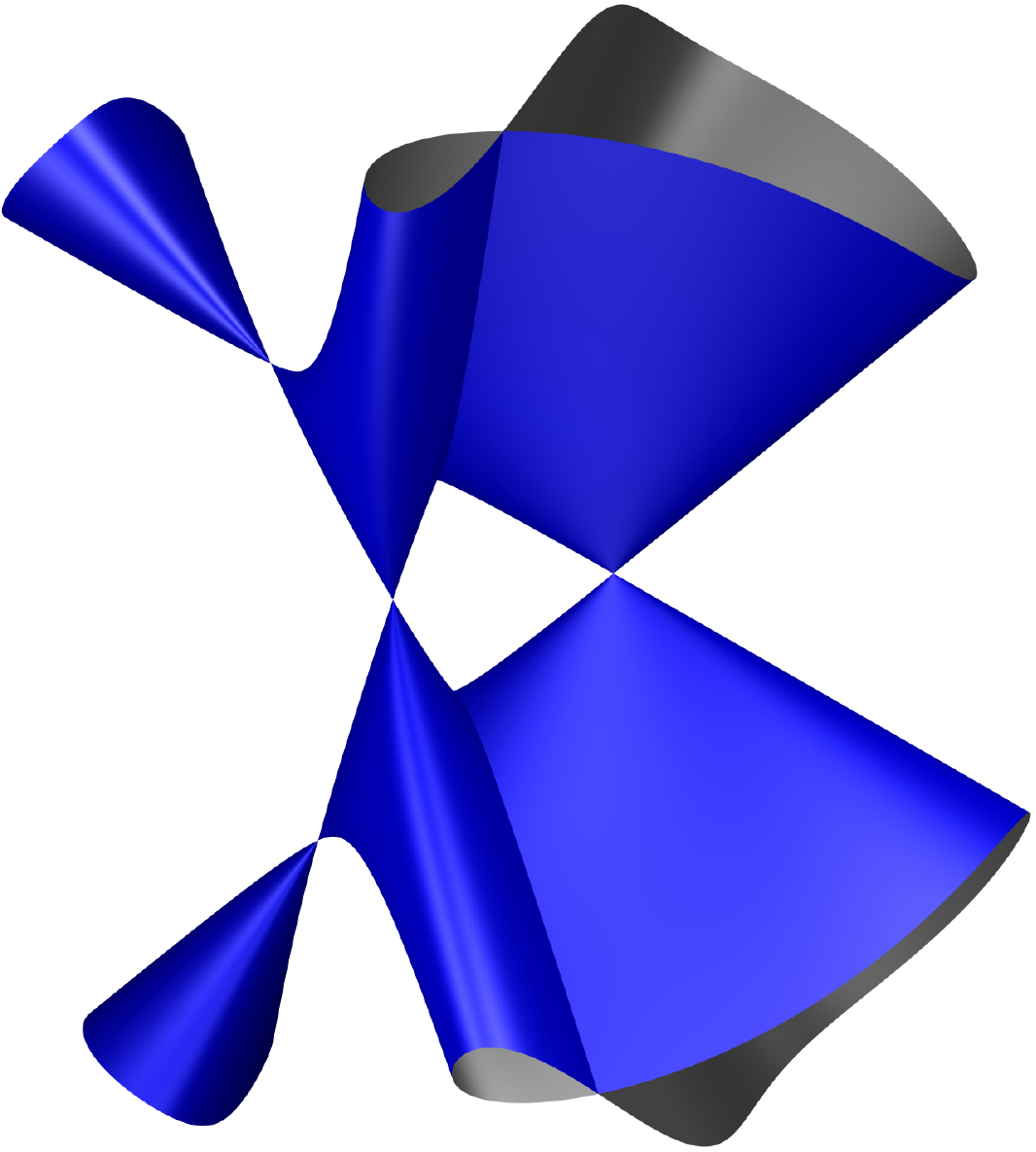}
    \caption{A general quartic symmetroid with a double smooth conic section. The surface has four real nodes.}
\end{figure}

\section{Quartic Symmetroids with a Double Twisted Cubic Curve}

\noindent
All quartic symmetroids with a double twisted cubic curve are reducible.

\begin{proposition}
    \label{prop:no-twisted-cubic}
    Let $S \subset \PP^3$ be an irreducible quartic surface with a double twisted cubic curve $T\mkern-4mu$. Then $S$ is not a symmetroid.
\end{proposition}

\begin{proof}
    Assume for contradiction that $S$ is a symmetroid. Then there are two cases: The points on $T$ are either generically \points{3} or they are all \points{2}.
    
    Suppose that there are only finitely many \points{2} on $T\mkern-4mu$. Let $p_1, p_2 \in T$ be two \nodes{3} and assume that the associated quadrics of $p_1$ and $p_2$ have a common singular point $p$. \cref{lem:singularity-is-a-base-point} implies that $p$ is a base point for $W(S)$. All the associated quadrics on the line $L$ spanned by $p_1$ and $p_2$, are singular at $p$. By \cref{lem:singular-at-base-point}, $L$ is contained in $\Sing(S)$. But then $S$ is singular along a quartic curve, so $S$ is reducible, which is a contradiction. We conclude that the associated quadrics of $p_1$ and $p_2$ have different apexes. It follows from \cref{lem:singularity-is-a-base-point} that $T$ gives rise to a curve of base points for $W(S)$, but this is impossible by \cref{lem:curve-in-base-locus}.    
    
    Assume that $T$ consists of \points{2}. The secant variety of $T$ equals $\PP^3\mkern-4mu$. Hence every point in $\PP^3$ lies on either a secant line or a tangent line of $T\mkern-4mu$. Let $p$ be a point in $S \setminus T\mkern-4mu$. There exists a line $L$ through $p$ that meets $T$ in a scheme of length~$2$. Since $T$ is double on $S$, it follows from B{\'e}zout's theorem that $L$ is contained in~$S$. Thus $S$ is a scroll. The Jacobian ideal of $S$ defines $T$, and possibly some points. Along a general line $l$ in $S$, the Jacobian defines a scheme of length~$2$. However, $l$ corresponds to a pencil of quadrics with rank $2$ and $3$. \cref{lem:rank3-pencil} implies that $l$ contains a scheme of length $3$ of \points{2}. This contradicts the fact that the \locus{2} is contained in the singular locus. 
\end{proof}

\section{Quartic Symmetroids with a Triple Point}

\noindent
Let $S \subset \PP^3$ be a quartic surface with a triple point $p$. Note that the projection $\pi_p \colon S \setminus \{p\} \dashrightarrow \PP^2$ from $p$ is a birational map, so $S$ is rational. 

If $p \coloneqq [1 : 0 : 0 : 0]$, then the equation of $S$ can be written as $x_0F_3 + F_4$, where $F_3, F_4 \in \C[x_1, x_2, x_3]$ are polynomials of degree $3$ and $4$, respectively. The cubic cone $C \coloneqq \V(F_3)$ intersects $S$ in twelve lines, which meet at $p$. Let $p_1, \ldots, p_{12} \in \PP^2$ be the images of the these lines under $\pi_p$. Then $S$ can be represented as the image of the map induced by the linear system of quartic curves through the $p_i$. Let $e_i$ be the linear equivalence class of the exceptional line over the point $p_i$.

If $S$ is a symmetroid, then the matrix defining $S$ can be written as
\begin{align*}
    \setlength{\arraycolsep}{6pt}
    \begin{bmatrix}
        x_0 + l_{00} & l_{01} & l_{02} & l_{03} \\
              l_{01} & l_{11} & l_{12} & l_{13} \\
              l_{02} & l_{12} & l_{22} & l_{23} \\
              l_{03} & l_{13} & l_{23} & l_{33} \\
    \end{bmatrix}
    \mkern-7mu,
\end{align*}
where $l_{ij} \in \C[x_1, x_2, x_3]_1$ are linear forms. Moreover, $F_3$ is equal to the determinant of the submatrix
\begin{align*}
    \begin{bmatrix}
        l_{11} & l_{12} & l_{13} \\
        l_{12} & l_{22} & l_{23} \\
        l_{13} & l_{23} & l_{33}
    \end{bmatrix}
    \mkern-7mu.
\end{align*}
This implies that $C$ is tangent to $S$ along the sextic curve given by the zero locus of the $(3 \times 3)$-minors of the submatrix
\begin{align*}
    \begin{bmatrix}
        l_{01} & l_{11} & l_{12} & l_{13} \\
        l_{02} & l_{12} & l_{22} & l_{23} \\
        l_{03} & l_{13} & l_{23} & l_{33}
    \end{bmatrix}
    \mkern-7mu.
\end{align*}
Hence, the twelve lines on $S$ through $p$ coincide in such a way that they occur with even multiplicity. The general case is that two and two lines coincide. Then six of the linear equivalence classes $e_i - e_j$ are effective. This induces six nodes on $S$.

\begin{proposition}[{\cite[Article~93]{Jes16}}]
    \label{prop:general-triple-point}
    Let $S \subset \PP^3$ be a general quartic symmetroid with a triple point. Then $S$ has six additional nodes.
\end{proposition}

\noindent
Jessop proves that if $S$ is a quartic surface with a triple point and six nodes that do not lie on a conic section, then $S$ is a symmetroid. This fact makes it straightforward to deduce the size of the family of symmetroids with a triple point:

\begin{proposition}[{\cite[Article~93]{Jes16}}]
    \label{prop:dimension-triple}
    The family of quartic symmetroids with a triple point is $21$-dimensional. 
\end{proposition}

\begin{figure}[hbtp]
    \centering
    \includegraphics[height = 0.28\textheight]{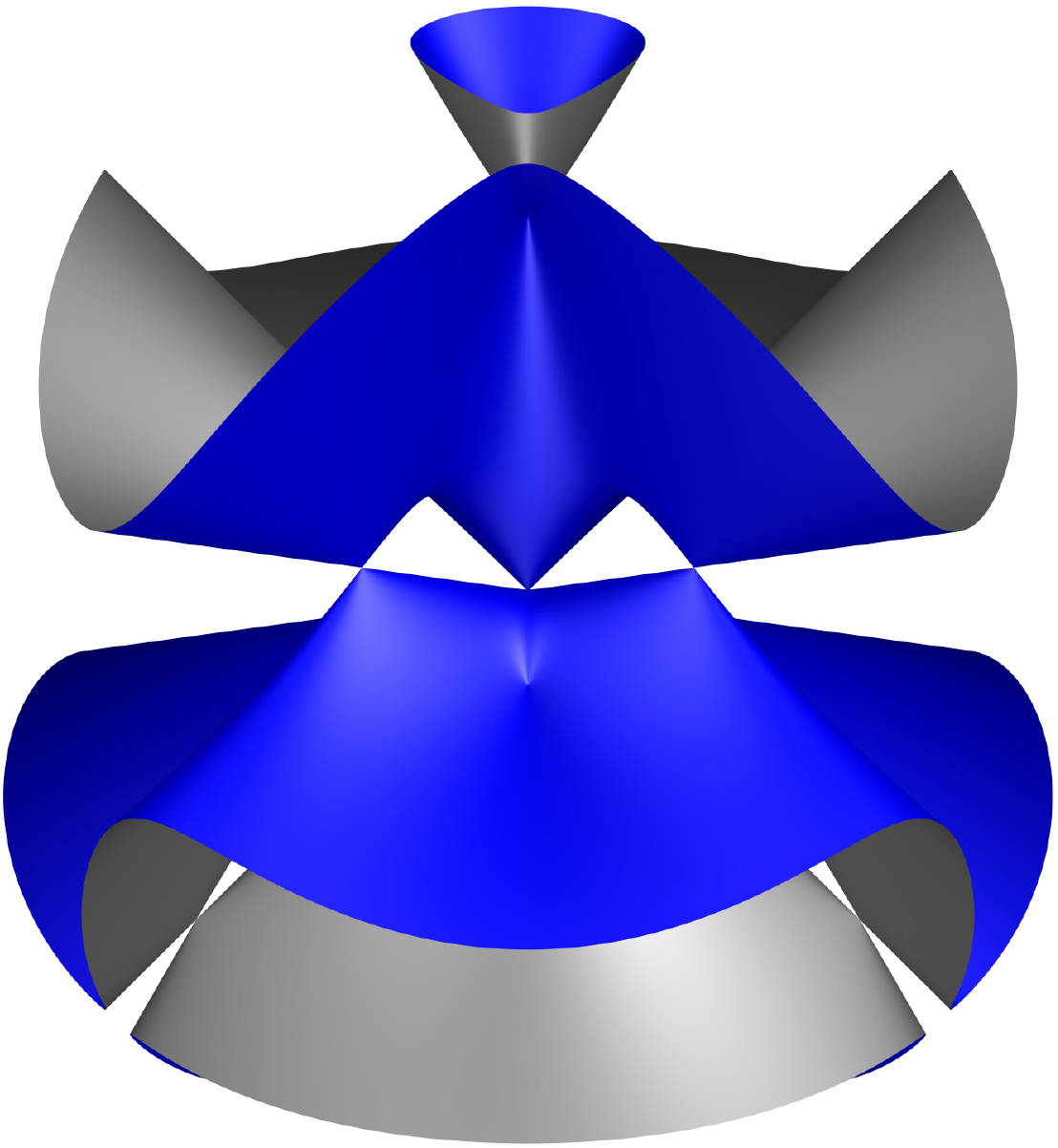}
    \caption{A general quartic symmetroid with a triple point, which is the central singularity. The surface has six real nodes.}
\end{figure}

\section{Quartic Symmetroids with an Elliptic Double Point}

\noindent We say that a double point is \emph{elliptic} if there is a curve of arithmetic genus~$1$ with support on the exceptional curve of a minimal resolution of the singularity. In \cite{Noe89}, Noether describes three classes, $S^{(1)}_4\mkern-4mu$, $S^{(2)}_4$ and $S^{(3)}_4\mkern-4mu$, of rational quartic surfaces having an elliptic double point. He expresses them by linear systems of plane curves and gives explicit equations for the surfaces. He proves that together with the quartic surfaces having a double curve or a triple point, these are the only rational quartic surfaces.

We show that of these types, only $S^{(1)}_4$ can occur as a symmetroid. The singularity of type $S^{(1)}_4$ is called a \emph{tacnode}. Moreover, we prove that a general tacnodal symmetroid has six additional nodes.

The rational parametrisation of these surfaces is given by linear systems of plane curves passing through some base points $p_i$. Let $e_i$ be the linear equivalence class of the exceptional line over the point $p_i$, and let $l$ be the class of the pullback of a line in $\PP^2\mkern-4mu$. The linear systems can then be expressed as
\begin{flalign*}
    & S^{(1)}_4\colon 
    & \mathclap{6l - 2\sum_{i = 1}^{7} e_i - \sum_{i = 8}^{11} e_i,} &&
    \\[0.7ex]
    & S^{(2)}_4\colon 
    & \mathclap{7l - 3e_1 - 2\sum_{i = 2}^{10} e_i,} &&
    \\
    & S^{(3)}_4\colon
    & \mathclap{9l - 3\sum_{i = 1}^{8} e_i - 2e_9 - e_{10}.} &&
\end{flalign*}
In all three cases, the base points $p_i$ lie on a cubic curve.

Choosing coordinates such that the elliptic double point is $p \coloneqq [1 : 0 : 0 : 0]$, we may assume that polynomials defining the different types of surfaces, are on the forms
\begin{flalign}
    \label{eq:type1}
    & S^{(1)}_4\colon 
    &\mathclap{x_1^2x_0^2 + (x_1F_2)x_0 + F_4,} &&
    \\
    \label{eq:type2}
    & S^{(2)}_4\colon 
    &\mathclap{x_1^2x_0^2 + (x_1x_3(2x_3 + B_1) + A_3)x_0 
    + x_3^4 + B_1x_3^3 + B_2x_3^2 + B_3x_3 + B_4,} &&
    \\
    \label{eq:type3}
    & S^{(3)}_4\colon
    &\mathclap{x_1^2x_0^2 + 2(x_1x_3A_1 + A_3)x_0 - x_1x_3^3 + B_2x_3^2 
    + B_3x_3 + B_4,} &&
\end{flalign}
where $F_d$ is a form of degree~$d$ in $\C[x_1, x_2, x_3]$ and $A_d$, $B_d$ are forms of degree~$d$ in $\C[x_1, x_2]$.

If $S$ is a quartic surface with a double point at $p$, then we may write its equation as 
\begin{align}
    \label{eq:quartic-double-point}
    F_2x_0^2 + F_3x_0 + F_4,
\end{align}
where $F_d$ is a form of degree $d$ in $\C[x_1, x_2, x_3]$. If $S$ is a symmetroid, the quadric $\V(F_2)$ is called the \emph{tangent cone} of $S$ at $p$. In the cases $S^{(1)}_4\mkern-4mu$, $S^{(2)}_4$ and $S^{(3)}_4\mkern-4mu$, the tangent cone is a double plane. Using this fact, we show that elliptic double points have rank $2$:

\begin{lemma}
    \label{lem:elliptic-rank2}
    Let $S$ be a general rational quartic symmetroid with an elliptic double point $p$. Then $p$ is a \point{2}.
\end{lemma}

\begin{proof}
    Note that the generality of $S$ means, in particular, that $S$ is not singular along a curve or has a triple point.

    Assume for contradiction that $p$ is a \singularity{3} and choose coordinates such that $p \coloneqq [1 : 0 : 0 : 0]$. Then we may assume, as in the proof of \cref{lem:singularity-is-a-base-point}, that the matrix defining $S$ is
    \begin{align*}
        M
        =
        \begin{bmatrix}
            a_0x_0 + l_{00} &     l_{01}      &     l_{02}      & l_{03} \\
                l_{01}      & a_1x_0 + l_{11} &     l_{12}      & l_{13} \\
                l_{02}      &     l_{12}      & a_2x_0 + l_{22} & l_{23} \\
                l_{03}      &     l_{13}      &     l_{23}      &    0
        \end{bmatrix}
        \mkern-7mu,
    \end{align*}
    where $l_{ij} \in \C[x_1, x_2, x_3]_1$ and $a_0, a_1, a_2 \in \C\setminus\{0\}$. Then
    \begin{align*}
        \det(M)
        =
        - 
        \big(
            a_1a_2l_{03}^2 + a_0a_2l_{13}^2 + a_0a_1l_{23}^2
        \big)
        x_0^2
        +
        F_3x_0
        +
        F_4,
    \end{align*}
    where $F_d$ is a form of degree $d$ in $\C[x_1, x_2, x_3]$. Since $p$ is an elliptic double point, then $-\big(a_1a_2l_{03}^2 + a_0a_2l_{13}^2 + a_0a_1l_{23}^2 \big)$ is a square. By scaling the $l_{i3}$, we have that
    \begin{align*}
        l_{03}^2 + l_{13}^2 + l_{23}^2 = l^2
    \end{align*}
    for some linear form $l$. Equivalently,
    \begin{align*}
        (l_{03} - il_{13})(l_{03} + il_{13}) = (l - l_{23})(l + l_{23}).
    \end{align*}
    Since $\C[x_1, x_2, x_3]$ is a unique factorisation domain, it follows that $l_{23}$ is a linear combination of $l_{03}$ and $l_{13}$.

    Note that every term in $\det(M)$ has an $l_{i3}l_{j3}$ factor. Because $l_{23}$ vanishes whenever both $l_{03}$ and $l_{13}$ vanish, $S$ is singular along the line $\V(l_{03}, l_{13})$. This contradicts the generality of $S$.
\end{proof}

\subsection{Type \texorpdfstring{$\boldsymbol{S^{(1)}_4}\mkern-4mu$}{S(1)4}}

Tacnodal surfaces are distinguished from types $S^{2}_4$ and $S^{3}_4$ by the intersection of the reduced tangent cone at the elliptic double point, with the surface.

\begin{lemma}
    \label{lem:two-double-lines}
    Let $S \subset \PP^3$ be a general, irreducible quartic symmetroid with a tacnode~$p$. Then the reduced tangent cone of $S$ at $p$, intersects $S$ in two double lines.
\end{lemma}

\begin{proof}
    In \eqref{eq:type1}, the reduced tangent cone at $p$ is the plane $H \coloneqq \V(x_1)$. It is clear from \eqref{eq:type1} that $H$ intersects $S$ in a cone, that is, four concurrent lines, $L_1$, $L_2$, $L_3$ and $L_4$.
    When $S$ is a symmetroid, we interpret $H$ as a net $N$ of quadrics, where the discriminant~$D$ consists of four pencils $P_i$, corresponding to $L_i$.
    Let $Q$ be the quadric satisfying $[Q] = p$.

    By \cref{lem:rank3-pencil}, each $P_i$ is of one of two types. We say that $P_i$ is of type I if the quadrics in $P_i$ have a common singularity, and that $P_i$ is of type II if the quadrics have a common tangent plane.

    First, we show that at most one of the pencils are of type I. Suppose that $P_1$ and $P_2$ are of type I. If all the \quadrics{2} in $P_1$ coincide, then the base locus of $P_1$ contains a triple line. It follows that the line $\Sing(Q)$ is contained in the base locus of $P_1$. If the \quadrics{2} in both $P_1$ and $P_2$ coincide, then $\Sing(Q)$ is in the base locus of $N\mkern-2mu$. \cref{lem:square-discriminant} implies that $D$ is a square. On the other hand, if $P_1$ contains at least one \quadric{2} $Q'$ different from $Q$, then $[Q']$ is a singular point on $S$. Any line $L$ in $H$ through $[Q']$ meets $L_2$, $L_3$ and $L_4$ in a point each. It follows that $L$ is contained in $S$ and thus that $H$ is a component in $S$. This is impossible. Hence $P_1$ must coincide with one of the other pencils.

    Next, we show that at most two of the pencils are of type II. If $P_1$ is of type II, then one of the planes in $Q$ is the common tangent plane for the quadrics in $P_1$. It follows that if three or more of the $P_i$ are of type II, then there are two pencils, $P_1$ and $P_2$, having the same common tangent plane $H'\mkern-4mu$. Choose coordinates such that $H'$ is $\V(y_3)$ and such that $H'$ is tangent to the quadrics in $P_i$ along the line $\V(y_i, y_3)$, for $i = 1, 2$. Then the quadrics in $P_1$ contain no $y_0^2$, $y_0y_1$, $y_0y_2$, $y_1y_2$ or $y_2^2$ terms. Likewise, the quadrics in $P_2$ contain no $y_0^2$, $y_0y_1$, $y_0y_2$, $y_1y_2$ or $y_1^2$ terms. Hence the quadrics in $N$ have no $y_0^2$, $y_0y_1$, $y_0y_2$ or $y_1y_2$ terms. Thus $N$ is contained in the space of quadrics parametrised by the matrix
    \begin{align*}
        A
        \coloneqq
        \begin{bmatrix}
              0    &   0    &   0    & x_{03} \\
              0    & x_{11} &   0    & x_{13} \\
              0    &   0    & x_{22} & x_{23} \\
            x_{03} & x_{13} & x_{23} & x_{33}
        \end{bmatrix}
        \mkern-7mu.
    \end{align*}
    The determinant is $\det(A) = -x_{03}^2 x_{11} x_{22}$. Hence the discriminant is not reduced. Moreover, note that the pencils of quadrics defined by $\V(x_{03})$, $\V(x_{11})$ and $\V(x_{22})$ each have a common singular point at $[1 : 0 : 0 : 0]$, $[x_{13} : -x_{03} : 0 : 0]$ and $[x_{23} : 0 : -x_{03} : 0]$, respectively. Thus they are all of type I, which is a contradiction.

    The discriminant $D$ consists of four lines, but the above paragraphs show that we can have at most three distinct pencils. In order to make sense of this, we consider pencils appearing with higher multiplicity. Consider the case where both $P_1$ and $P_2$ are of type II, whereas $P_3$ and $P_4$ coincide and are of type I. Let $H_1$ and $H_2$ be the common tangent planes of $P_1$ and $P_2$, respectively. The base locus of $P_1$ consists of a double line $l_1 \subset H_1$ and a conic section $C_1 \subset H_2$. Similarly, the base locus of $P_2$ consists of a double line $l_2 \subset H_2$ and a conic section $C_2 \subset H_1$. Let $p' \in H_1 \cap H_2$ be the common singularity of $P_3$. Since $P_3$ appears twice in $D$, \cref{lem:singularity-is-a-base-point} implies that $p'$ is a base point for $N\mkern-2mu$. Let $p'' \in l_1 \cap C_2$ be an intersection point different from $p'\mkern-3mu$. Then $p''$ is a base point for $N\mkern-2mu$. The line $L$ spanned by $p'$ and $p''$ lies in $H_1$, so it is tangent to all quadrics in $P_1$. Since the quadrics in $P_1$ pass through both $p'$ and $p''\mkern-3mu$, they contain $L$. Moreover, $L$ is in the base locus of $P_3$. Thus $L$ is in the base locus of $N\mkern-2mu$. \cref{lem:square-discriminant} states that $D$ is a square.

    Finally, if a pencil of type II appears with higher multiplicity in $D$, then \cref{lem:singularity-is-a-base-point} implies that the base locus of $N$ contains the line of singular points. By \cref{lem:square-discriminant}, $D$ is a square.
\end{proof}

\begin{lemma}
    \label{lem:tacnode-basepoints}
    Let $S \subset \PP^3$ be a general, irreducible quartic symmetroid with a tacnode~$p$. Then $W(S)$ has two base points.
\end{lemma}

\begin{proof}
    By \cref{lem:two-double-lines}, the reduced tangent cone at $p$ intersects $S$ in two double lines. This corresponds to a net~$N$ of quadrics, where the discriminant~$D$ consists of two double lines, $L_1$ and $L_2$. Because $S$ is generic, the general quadric in each of these pencils has rank $3$. \cref{lem:rank3-pencil} states that the quadrics along $L_i$ have either a common singular point or the singular points form a line. By \cref{lem:singularity-is-a-base-point}, $L_i$ gives rise to a single base point or a line of base points for $N\mkern-2mu$, respectively.

    Suppose that $L_1$ gives rise to a line of base points. Let $Q$ be a \quadric{3} that corresponds to a point $[Q] \in L_2$ not contained in $L_1$. Let $q$ be the singular point of $Q$, and let $q'$ be one of the base points coming from $L_1$. Then the line spanned by $q$ and $q'$ is contained in $Q$. It follows that the plane spanned by $q$ and the line of base points, is contained in $Q$. This contradicts the assumption that $Q$ has rank~$3$. We conclude that the quadrics along each line $L_i$ have a common singularity.

    Let $q_i$ be the common singularity of the quadrics along $L_i$. Then the line $L$ spanned by $q_1$ and $q_2$, is contained in the base locus of $N\mkern-2mu$. Extend $N$ with a quadric~$Q' \notin N\mkern-2mu$, such that $N$ and $Q'$ span the web $W(S)$. Then $Q'$ intersects $L$ in two points. Hence $W(S)$ has two base points.
\end{proof}

\begin{proposition}
    \label{prop:general-tacnode}
    Let $S \subset \PP^3$ be a general quartic symmetroid with a tacnode~$p$. Then $S$ has six additional nodes.
\end{proposition}

\begin{proof}
    Following \cref{lem:tacnode-basepoints}, consider the $\PP^7$ of quadrics passing through the base points $p_1$ and $p_2$. We shall describe the \locus{2} of this space. It consists of two components, $X_1$ and $X_2$. First, let $H_{12}$ be a plane containing both $p_1$ and $p_2$, and let $H$ be any plane in $\PP^3\mkern-4mu$. The union $H_{12} \cup H$ is a \quadric{2} passing through the base points. The set $X_1 \subset \PP^7$ of all such unions is a fourfold of degree~$4$. Next, let $H_1$ be a plane containing $p_1$, and $H_2$ a plane containing $p_2$. The union $H_1 \cup H_2$ is a \quadric{2} passing through the base points. The set $X_2 \subset \PP^7$ of all such unions is a fourfold of degree~$6$.

    Let $H_{12}$ and $H'_{12}$ be two planes that both contain $p_1$ and $p_2$. Then both  $X_1$ and $X_2$ are singular at the point $[H_{12} \cup H'_{12}]$. From the proof of \cref{lem:tacnode-basepoints}, it is clear that $p_1$ and $p_2$ are contained in the singular locus of the quadric associated to the tacnode~$p$. Thus this quadric consists of two planes that both contain the base points. Hence, $W(S) \subset \PP^7$ intersects $\Sing(X_1) \cap \Sing(X_2)$.

    Choose coordinates such that $p_1 \coloneqq [1 : 0 : 0 : 0]$ and $p_2 \coloneqq [0 : 1 : 0 : 0]$. A general plane $H$ is given by $a_0x_0 + a_1x_1 + a_2x_2 + a_3x_3 = 0$ and a plane $H_{12}$ through $p_1$ and $p_2$ is given by $b_2x_2 + b_3x_3 = 0$. Consider the Segre embedding $\sigma_{3, 1} \colon \PP^3 \times \PP^1 \to \PP^7$ given by
    \begin{align*}
        ([a_0 : a_1 : a_2 : a_3], [b_2 : b_3])
        \mapsto
        [a_0b_2 : a_1b_2 : a_2b_2 : a_3b_2 : a_0b_3 : a_1b_3 : a_2b_3 : a_3b_3].
    \end{align*}
    Let $\PP^7$ have coordinates $[x_{02} : x_{12} : x_{22} : x_{32} : x_{03} : x_{13} : x_{23} : x_{33}]$. The image $\Sigma_{3,1}$ of $\sigma_{3,1}$ is then given by the $(2 \times 2)$-minors of the matrix
    \begin{align*}
        M
        \coloneqq
        \begin{bmatrix}
            x_{02} & x_{12} & x_{22} & x_{32} \\
            x_{03} & x_{13} & x_{23} & x_{33}
        \end{bmatrix}
        \mkern-7mu.
    \end{align*}
    We can expand $M$ into a $(4 \times 4)$-matrix in the following manner:
    \begin{align*}
        \begin{bmatrix}
            x_{02} & x_{12} & x_{22} & x_{32} \\
            x_{03} & x_{13} & x_{23} & x_{33}
        \end{bmatrix}
        \leadsto
        \begin{bmatrix}
                   &        &        &        \\
                   &        &        &        \\
            x_{02} & x_{12} & x_{22} & x_{32} \\
            x_{03} & x_{13} & x_{23} & x_{33}
        \end{bmatrix}
        \leadsto
        \begin{bmatrix}
              0    &   0    & x_{02} & x_{03} \\
              0    &   0    & x_{12} & x_{13} \\
            x_{02} & x_{12} & x_{22} & x_{32} \\
            x_{03} & x_{13} & x_{23} & x_{33}
        \end{bmatrix}
        \eqqcolon
        A.
    \end{align*}
    Let $A'$ be defined as $A$ with $x_{23} = x_{32}$. Then the $(3 \times 3)$-minors of $A'$ define $X_1$. We deduce that $X_1$ is the Segre variety $\Sigma_{3, 1}$ projected down to $\PP^6\mkern-4mu$.

    We see from $A'$ that the base locus of the $\PP^6$ spanned by $X_1$, contains a line. By \cref{lem:curve-in-base-locus}, we may therefore assume that $W(S)$ is not contained in this $\PP^6\mkern-4mu$. Let $W'$ be the plane defined as the intersection of $W(S)$ and the hyperplane spanned by $X_1$. The web $W(S)$ is a generic $3$-space that is such that $W'$ meets $\Sing(X_1)$ in a point~$p$. The fourfold $X_1$ is singular precisely at points that correspond to the union of two planes that both contain the base points $p_1$ and $p_2$. Set-theoretically, $\Sing(X_1)$ is given by
\begin{align}
    \label{eq:singular-locus}
    x_{02} = x_{03} = x_{12} = x_{13} = 0.
\end{align}
    The only minor of $M$ that survives under the relations \eqref{eq:singular-locus} is
\begin{align*}
    \begin{vmatrix}
        x_{22} & x_{32} \\
        x_{23} & x_{33}
    \end{vmatrix}
    =
    x_{22}x_{33} - x_{23}x_{32}.
\end{align*}
    Let $Q \subset \PP^7$ be the quadric defined by this minor. The intersection of $Q$ and the linear space $V$ given by \eqref{eq:singular-locus}, is mapped two-to-one onto $\Sing(X_1)$ under the projection $\PP^7 \to \PP^6\mkern-4mu$. Let $\widetilde{W}' \subset \PP^7$ be the $3$-space lying over $W'\mkern-4mu$. Since the degree of $\Sigma_{3,1}$ is $4$, B{\'e}zout's theorem implies that $\widetilde{W}'$ meets $\Sigma_{3,1}$ in four points. The intersection $\widetilde{W}' \cap V$ is a line $L$ lying over $p$. The quadric $Q$ intersects $L$ in two points. Therefore, $\widetilde{W}'$ meets $\Sigma_{3, 1}$ in two points outside of $V$. It follows that $W'$ meets $X_1$ in two points outside of $p$. Thus the same is true for $W(S)$.

    An analogous argument for $X_2$, considering the Segre variety $\Sigma_{2, 2} \subset \PP^8\mkern-4mu$, shows that $W(S)$ generically intersects $X_2$ in four points outside of $p$. In total, $W(S)$ contains in general six \points{2} in addition to $p$.
\end{proof}

\begin{remark}
    The reader may wonder why we in the proof of \cref{prop:general-tacnode} argue via Segre varieties, instead of using B{\'e}zout's theorem directly on $X_1$ and $X_2$. The reason is that $X_1$ and $X_2$ are not normal. Using the description of $X_1$ as the vanishing of the $(3 \times 3)$-minors of $A'$, we calculate that $\Sing(X_1)$ is given as
    \begin{align*}
        \V\big( x^2_{02}, x_{02}x_{03}, x_{02}x_{12}, x_{02}x_{13}, x^2_{03}, x_{03}x_{12}, x_{03}x_{13}, x^2_{12}, x_{12}x_{13}, x^2_{13} \big)
    \end{align*}
    in the $\PP^6$ spanned by $X_1$. Thus $\Sing(X_1)$ is the whole first-order infinitesimal neighbourhood of the plane $\V(x_{02}, x_{03}, x_{12}, x_{13})$. This has codimension~$4$, hence its degree is $5$. The intersection $W' \cap \Sing(X_1)$ contains the whole first-order infinitesimal neighbourhood of $p$ in $W'\mkern-4mu$. This has codimension $2$ in $W'\mkern-3mu$. Hence $p$ appears in $W(S) \cap X_1$ with multiplicity $3$, not $2$.
\end{remark}

\noindent
We can now determine the number of symmetroids with a tacnode:

\begin{proposition}
    \label{prop:dimension-tacnode}
    The family of quartic symmetroids in $\PP^3$ with a tacnode is $20$-dimensional.
\end{proposition}

\begin{proof}
    As in the proof of \cref{prop:general-tacnode}, given two base points $p_1$ and $p_2$, a symmetroid with a tacnode corresponds to a $\PP^3 \subset \PP^7$ that intersects the plane $\Sing(X_1) \cap \Sing(X_2)$ in a point. It follows from \eqref{eq:dim-fixed-intersection} that there is a $14$-dimensional family of such $3$-spaces. A choice of base points $p_1$ and $p_2$ corresponds to a point in the $6$-dimensional space $\PP^3 \times \PP^3$. Hence the family of quartic symmetroids with a tacnode has in total dimension $14 + 6 = 20$.
\end{proof}

\begin{figure}[htbp]
    \centering
    \begin{minipage}{0.5\textwidth-1ex}
        \centering
        \includegraphics[height = 0.28\textheight]{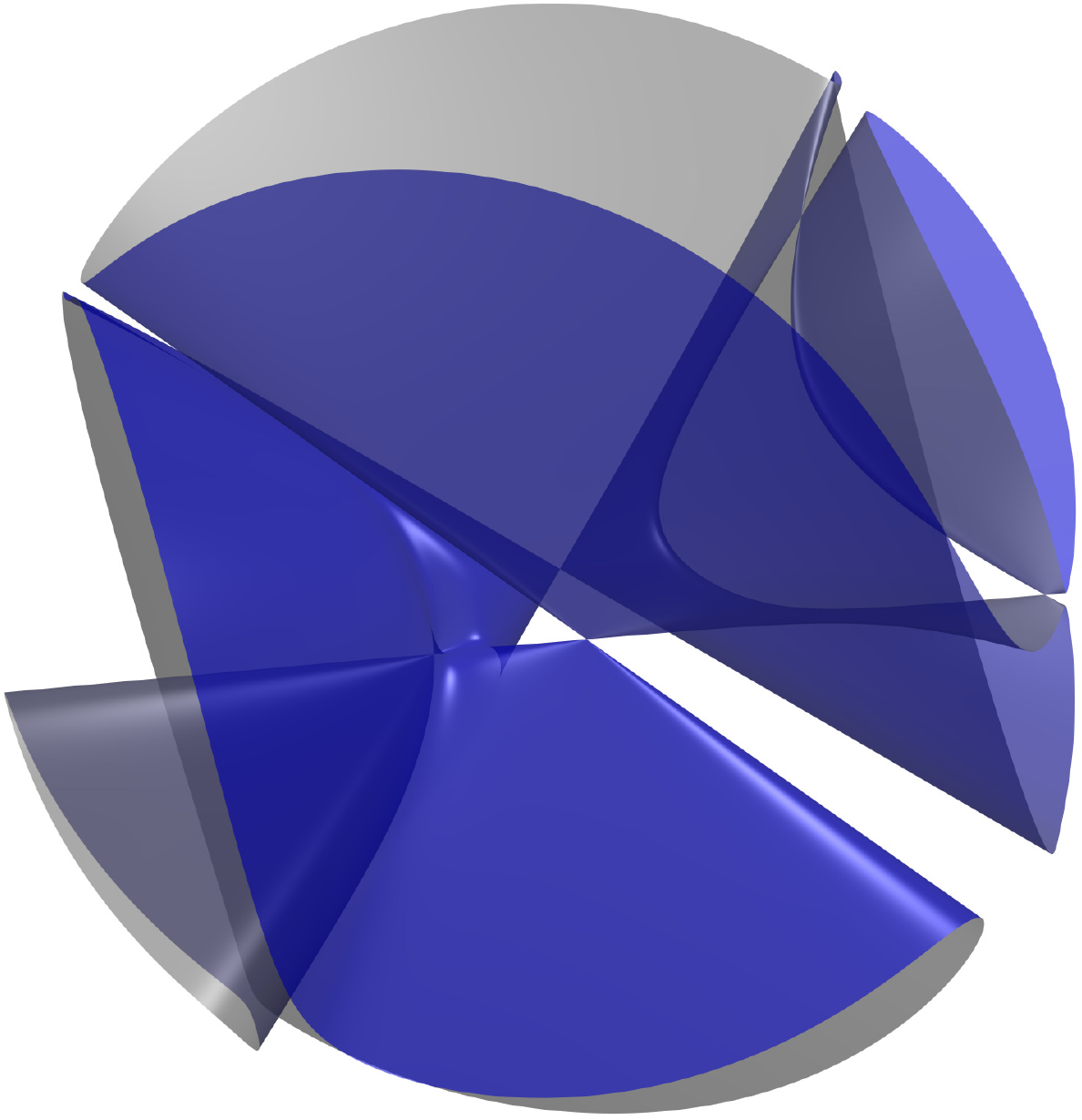}
    \end{minipage}
    \begin{minipage}{0.5\textwidth-1ex}
        \centering
        \includegraphics[height = 0.28\textheight]{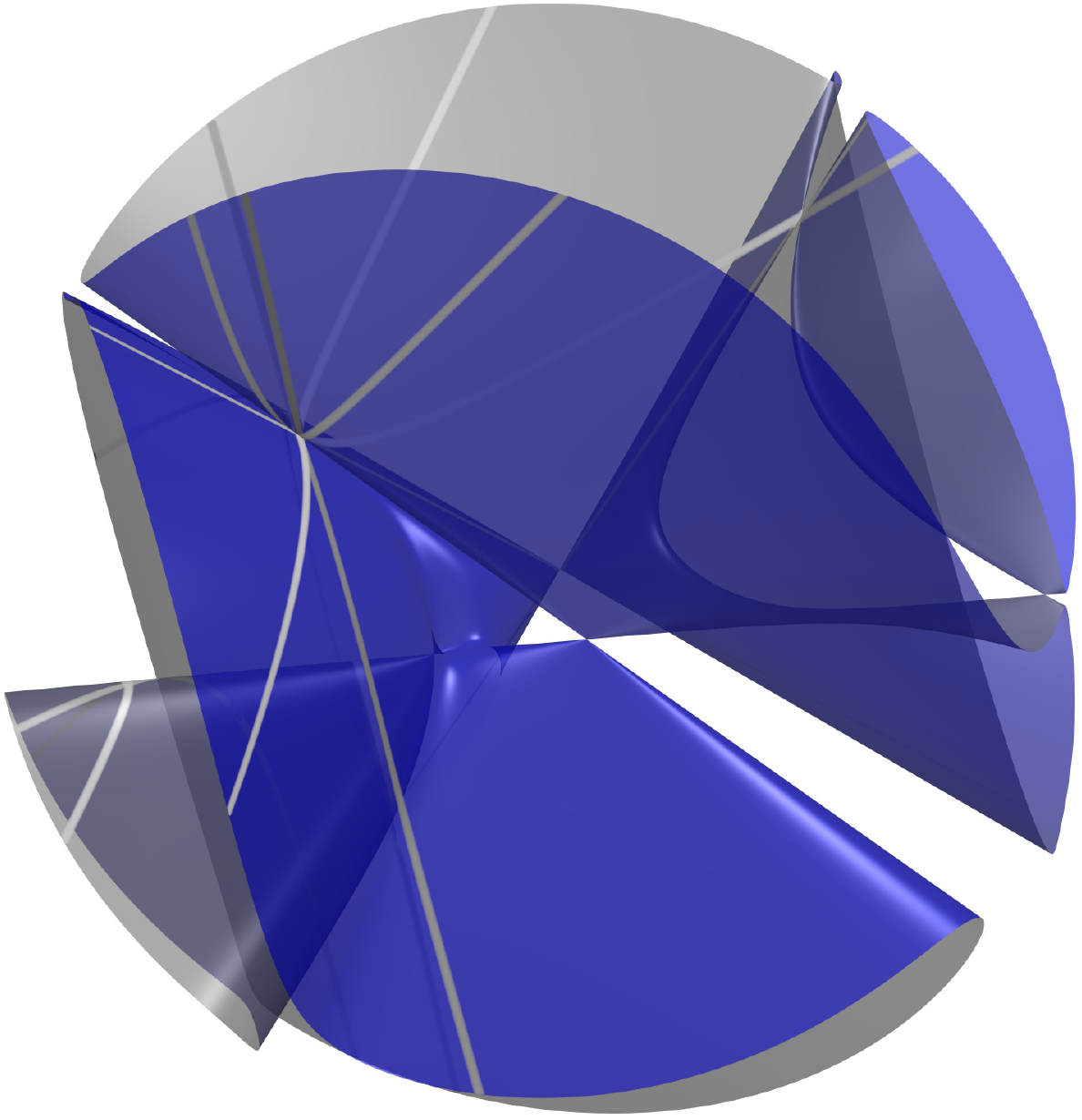}
    \end{minipage}
    \caption{A general quartic symmetroid with a tacnode. The figure to the right shows three plane sections meeting at the tacnode. The surface has six real nodes.}
\end{figure}

\subsection{Types \texorpdfstring{$\boldsymbol{S^{(2)}_4}$}{S(2)4} and \texorpdfstring{$\boldsymbol{S^{(3)}_4}\mkern-4mu$}{S(3)4}}

Hitherto, we have primarily argued in terms of the associated linear system of quadrics, but now we turn to ramification. This is a classical tool in the subject of surfaces.

Let $S \coloneqq \V(F)$ be a quartic surface and $p$ a double point on $S$. The projection from $p$ defines a two-to-one map, which extends to a morphism $\pi_p \colon \widetilde{S} \to \PP^2$ on the blow-up $\widetilde{S}$ of $S$ with centre $p$. If $F$ is as in \eqref{eq:quartic-double-point}, then $\pi_p$ is ramified along the sextic curve $R_p \coloneqq \V\big(F_3^2 - 4F_2F_4\big)$.
The following is one of the earliest results about symmetroids. For a modern proof and an extension, see \cite[Theorem~1.2]{Ott+14}.

\begin{theorem}[\cite{Cay69}]
    \label{thm:ramification}
    If $p$ is a \point{2} on a quartic symmetroid $S \subset \PP^3\mkern-4mu$, then the ramification locus $R_p = R_1 \cup R_2$ is the union of two cubic curves, $R_1$ and~$R_2$. 
\end{theorem}

\noindent
The next idea is taken from the proof of \cite[Theorem~1.2]{Ott+14}:

{
    \pretolerance = 124
    \begin{lemma}
        \label{lem:not-on-a-conic}
        Let $S \coloneqq \V\big(F_2x_0^2 + F_3x_0 + F_4\big) \subset \PP^3$ be an irreducible quartic symmetroid with a \point{2} at $p \coloneqq [1 : 0 : 0 : 0]$, where $F_d$ is a form of degree $d$ in $\C[x_1, x_2, x_3]$. The ramification locus $R_p$ splits into two cubics $\V(r_1)$ and $\V(r_2)$. Then $\V(r_1, F_4) = 2Z$ is two times a scheme $Z$ of length $6$, and $Z$ is \emph{not} contained in a conic section.
    \end{lemma}
}

\begin{proof}
    After conjugating with an appropriate matrix, we may assume that the matrix defining $S$ is
    \begin{align*}
        \begin{bmatrix}
              l_{00}     & x_0 + l_{01} & l_{02} & l_{03} \\
            x_0 + l_{01} &   l_{11}     & l_{12} & l_{13} \\
              l_{02}     &   l_{12}     & l_{22} & l_{23} \\
              l_{03}     &   l_{13}     & l_{23} & l_{33}
        \end{bmatrix}
        \mkern-7mu,
    \end{align*}
    where the $l_{ij}$ are linear forms in $\C[x_1, x_2, x_3]$. Then $F_4$ is the determinant of
    \begin{align*}
        M
        \coloneqq
        \begin{bmatrix}
            l_{00} & l_{01} & l_{02} & l_{03} \\
            l_{01} & l_{11} & l_{12} & l_{13} \\
            l_{02} & l_{12} & l_{22} & l_{23} \\
            l_{03} & l_{13} & l_{23} & l_{33}
        \end{bmatrix}
        \mkern-7mu.
    \end{align*}
    Also, $r_1$ is the $(0, 0)$-minor and $r_2$ the $(1, 1)$-minor of $M\mkern-2mu$.

    The fact that $\V(r_1, F_4)$ is a double scheme $2Z$ can be seen from the identity $r_1r_2 = F_3^2 - 4F_2F_4$, which shows that $F_2F_4$ is a square modulo $r_1$. We claim that $Z$ is equal to the scheme $Z'$ given by the $(3 \times 3)$-minors of the submatrix
    \begin{align*}
        A
        \coloneqq
        \begin{bmatrix}
            l_{01} & l_{11} & l_{12} & l_{13} \\
            l_{02} & l_{12} & l_{22} & l_{23} \\
            l_{03} & l_{13} & l_{23} & l_{33}
        \end{bmatrix}
        \mkern-7mu.
    \end{align*}
    Clearly, $Z' \subseteq \V(r_1, F_4)$. By the symmetry of $M\mkern-2mu$, it follows that $2Z' \subseteq \V(r_1, F_4)$. Equality follows by considering degrees. Thus $Z = Z'\mkern-3mu$.

    Assume for contradiction that $Z$ is contained in a conic section. Then $Z = \V(q, g)$ for a quadratic form $q$ and a cubic form $g$. The vector space of cubics vanishing on $Z$ is $\langle x_1q, x_2q, x_3q, g \rangle$. Considering the syzygies between the generators, we construct the Hilbert--Burch matrix
    \begin{align*}
        A'
        \coloneqq
        \begin{bmatrix}
            x_2 &           -x_1 &   0  & 0 \\
            x_3 &             0  & -x_1 & 0 \\
             0  & \phantom{-}x_3 & -x_2 & 0
        \end{bmatrix}
        \mkern-7mu.
    \end{align*}
    The $(3 \times 3)$-minors of $A'$ are identically zero. The matrix $A$ is row equivalent with $A'\mkern-3mu$, so $F_4 = \det(M) = 0$. This contradicts the assumption that $S$ is irreducible.
\end{proof}

\noindent
We show that if surfaces of type $S^{(2)}_4$ or $S^{(3)}_4$ satisfy the conclusions of \cref{thm:ramification,lem:not-on-a-conic}, then they degenerate to type $S^{(1)}_4\mkern-4mu$.

\begin{proposition}
    \label{prop:no-type2}
    Let $S$ be a rational quartic surface of type $S^{(2)}_4\mkern-4mu$. Then $S$ is not a symmetroid.
\end{proposition}

\begin{proof}
    We may assume that the equation defining $S$ is as in \eqref{eq:type2}. The ramification locus $R$ for the projection from the elliptic double point~$p$, is then given by
    \begin{align}
        \label{eq:type2-ramification}
        x_1 \big( x_1 \big(B_1^2 - 4B_2\big) + 4A_3 \big) x_3^2
        +
        x_1 (2A_3B_1 - 4x_1B_3) x_3
        +
        A_3^2 - 4x_1^2B_4.
    \end{align}
    The sextic curve $R$ has a quadruple point at $q \coloneqq [0 : 0 : 1]$.

    The following observation is used repeatedly throughout the proof: From \eqref{eq:type1} and \eqref{eq:type2}, we see that if $x_1$ divides $A_3$, then $S$ has a tacnode at $p$. This contradicts the assumption that $S$ is of type $S^{(2)}_4\mkern-4mu$. By \eqref{eq:type2-ramification}, this is equivalent to $\V(x_1)$ being a component of $R$. Assume that $x_1$ does not divide $A_3$.

    Assume for contradiction that $S$ is a symmetroid. \cref{lem:elliptic-rank2} and \cref{thm:ramification} imply that $R$ splits into two cubics, $R_1$ and $R_2$. Since $R$ has a quadruple point at $q$, there are two possibilities: Either both $R_1$ and $R_2$ have double points at $q$, or $R_1$ has a triple point and $R_2$ passes only once through $q$.

    \vskip0.5\baselineskip\noindent
    \emph{Case \emph{I} \dash $R_1$ and $R_2$ have double points at $q$.}
    \vskip0.5\baselineskip

    \noindent
    The equations for $R_1$ and $R_2$ can be written as $r_1 \coloneqq C_2x_3 + C_3$ and $r_2 \coloneqq D_2x_3 + D_3$, respectively, where $C_d$ and $D_d$ are forms of degree~$d$ in $\C[x_1, x_2]$. Since $R_1 \cup R_2 = R$, we have that $r_1r_2$ equals \eqref{eq:type2-ramification}. Equating the coefficients of the $x_3^n$ terms of $r_1r_2$ and \eqref{eq:type2-ramification}, produces the following system of equations:
    \begin{align}
        \label{eq:caseI-1}
        C_2 D_2 &= x_1 \big( x_1 \big(B_1^2 - 4B_2\big) + 4A_3 \big),
        \\
        \label{eq:caseI-2}
        C_2 D_3 + C_3 D_2 &= x_1 (2A_3B_1 - 4x_1B_3),
        \\
        \nonumber
        C_3 D_3 &= A_3^2 - 4x_1^2B_4.
    \end{align}
    From \eqref{eq:caseI-1}, we get that $x_1$ is a factor in either $C_2$ or $D_2$. Suppose that $x_1$ divides $C_2$. Inserting this into \eqref{eq:caseI-2}, we get that $x_1$ divides either $C_3$ or $D_2$. Suppose first that $x_1$ is a factor in $C_3$. Then $\V(x_1)$ is a component of $R_1$, so $\V(x_1)$ is a component of $R$, which implies that $S$ has a tacnode at $p$. Now suppose that $x_1$ is a factor in $D_2$. Inserting this back into \eqref{eq:caseI-1}, we find that $x_1$ divides $A_3$. Hence $S$ has a tacnode at~$p$.

    \vskip0.5\baselineskip\noindent
    \emph{Case \emph{II} \dash $R_1$ has a triple point and $R_2$ has a single point at $q$.}
    \vskip0.5\baselineskip

    \noindent
    The equations for $R_1$ and $R_2$ can be written as $r_1 \coloneqq C_1x_3^2 + C_2x_3 + C_3$ and $r_2 \coloneqq D_3$, respectively, where $C_d$ and $D_d$ are forms of degree~$d$ in $\C[x_1, x_2]$. We can assume that $x_1$ is not a factor in $D_3$, as that would imply that $\V(x_1)$ is a component of $R$. Equating the coefficients of the $x_3^n$ terms of $r_1r_2$ and \eqref{eq:type2-ramification}, produces the following system of equations:
    \begin{align}
        \label{eq:caseII-1}
        C_1 D_3 &= x_1 \big( x_1 \big(B_1^2 - 4B_2\big) + 4A_3 \big),
        \\
        \label{eq:caseII-2}
        C_2 D_3 &= x_1 (2A_3B_1 - 4x_1B_3),
        \\
        \label{eq:caseII-3}
        C_3 D_3 &= A_3^2 - 4x_1^2B_4.
    \end{align}
    Since $x_1$ does not divide $D_3$, we get from \eqref{eq:caseII-1} that, up to scalar, $C_1 = x_1$ and $D_3 = x_1 \big(B_1^2 - 4B_2\big) + 4A_3$. Similarly, \eqref{eq:caseII-2} yields that $C_2 = x_1C'_1$ for some linear form $C'_1 \in \C[x_1, x_2]$. Substituting for $C_2$ and $D_3$ into \eqref{eq:caseII-2} and cancelling $x_1$, we get
    \begin{align}
        \nonumber
        \big( x_1 \big(B_1^2 - 4B_2\big) + 4A_3 \big)C'_1
        &= 2A_3B_1 - 4x_1B_3,
        \shortintertext{which is equivalent to}
        \label{eq:caseII-4}
        x_1 \big( \big(B_1^2 - 4B_2 \big) C'_1 + 4B_3 \big)
        &= 2A_3 (B_1 - 2C'_1).
    \end{align}
    Either $A_3$ or $B_1 - 2C'_1$ is divisible by $x_1$, but we have assumed the former to be incorrect. Therefore, $C'_1 = \frac{1}{2}B_1 - ax_1$ for some $a \in \C\mkern-1mu$. Putting this back into \eqref{eq:caseII-4} and cancelling $x_1$, we obtain
    \begin{gather*}
        \big(B_1^2 - 4B_2 \big) \bigg( \frac{1}{2}B_1 - ax_1 \bigg) + 4B_3
        =
        4aA_3,
        \shortintertext{so}
        B_3 = aA_3 + \frac{1}{4}\big(B_1^2 - 4B_2 \big) \bigg( ax_1 -  \frac{1}{2}B_1 \bigg)\mkern-2mu.
    \end{gather*}
    Inserting the expression we found for $D_3$ into \eqref{eq:caseII-3} gives
    \begin{align}
        \nonumber
        \big( x_1 \big(B_1^2 - 4B_2\big) + 4A_3 \big) C_3 &= A_3^2 - 4x_1^2B_4,
        \shortintertext{which can be transformed to}
        \label{eq:caseII-5}
        x_1 \big( \big(B_1^2 - 4B_2\big)C_3 + 4x_1B_4) &= A_3 (A_3 - 4C_3).
    \end{align}
    Because $x_1$ is not a factor in $A_3$, we conclude that $A_3 - 4C_3$ is divisible by $x_1$. Thus, $C_3 = \frac{1}{4}A_3 - x_1C'_2$ for some quadratic form $C'_2 \in \C[x_1, x_2]$. Putting this into \eqref{eq:caseII-5} and cancelling $x_1$, we get
    \begin{gather}
        \nonumber
        \big(B_1^2 - 4B_2\big) \bigg(\frac{1}{4}A_3 - x_1C'_2 \bigg) + 4x_1 B_4
        =
        4A_3 C'_2,
        \shortintertext{which is equivalent to}
        \label{eq:caseII-6}
        x_1 \big(4B_4 - \big(B_1^2 - 4B_2\big)C'_2 \big)
        =
        A_3 \bigg(4C'_2 - \frac{1}{4} \big(B_1^2 - 4B_2 \big) \bigg)\mkern-2mu.
    \end{gather}
    Again, $A_3$ is not divisible by $x_1$, so $x_1$ is a factor in $4C'_2 - \frac{1}{4} \big(B_1^2 - 4B_2 \big)$. Hence $C'_2 = \frac{1}{16} \big(B_1^2 - 4B_2 \big) + x_1C''_1$ for some linear form $C''_1 \in \C[x_1, x_2]$. Substituting for $C'_2$ in \eqref{eq:caseII-6} and cancelling $x_1$, we obtain 
    \begin{gather*}
    4B_4 - \big(B_1^2 - 4B_2\big) \bigg( \frac{1}{16} \big(B_1^2 - 4B_2 \big) + x_1C''_1 \bigg) = 4A_3C''_1,
    \shortintertext{so}
    B_4 = A_3C''_1 + \frac{1}{64} \big(B_1^2 - 4B_2\big)^{\mkern-3mu2} + \frac{1}{4}x_1\big(B_1^2 - 4B_2\big)C''_1.
    \end{gather*}
    We have now described the necessary conditions on $B_3$ and $B_4$ for $R$ to split into cubics.
    
    Using the above expressions for $r_1$, $r_2$, $B_3$ and $B_4$, we can verify the relations
    \begin{align*}
        \big(8x_3^2 - 4(4ax_1 - B_1)x_3 - 16x_1C_1 - \big(B_1^2 - 4B_2\big) \big)^{\mkern-3mu2}
        &= 64f_4 - 256(ax_3 + C''_1)r_1,
        \\
        \big( 8x_3^2 + 4B_1x_3 - \big(B_1^2 - 4B_2 \big) \big)^{\mkern-3mu2}
        &= 64f_4 - 16(ax_3 + C''_1)r_2,
    \end{align*}
    where $f_4 \coloneqq x_3^4 + B_1x_3^3 + B_2x_3^2 + B_3x_3 + B_4$. Hence for $i = 1, 2$, we have that $\V(r_i, f_4) = 2Z_i$ for a scheme $Z_i$ that is contained in a conic. This contradicts \cref{lem:not-on-a-conic}, showing that $S$ is not a symmetroid.
\end{proof}

\begin{proposition}
    \label{prop:no-type3}
    Let $S$ be a rational quartic surface of type $S^{(3)}_4\mkern-4mu$. Then $S$ is not a symmetroid.
\end{proposition}

\begin{proof}
    We may assume that the equation defining $S$ is as in \eqref{eq:type3}. The ramification locus $R$ for the projection from the elliptic double point~$p$, is then given by
    \begin{align*}
        x_1^3 x_3^3 + x_1^2(A_1^2 - B_2)x_3^2 
        + x_1(2 A_1 A_3  - x_1B_3)x_3 + A_3^2 - x_1^2B_4.
    \end{align*}
    The sextic curve $R$ has two consecutive triple points \dash a triple point with an infinitely near triple point \dash at the point $q \coloneqq [0 : 0 : 1]$, with tangent direction $x_1 = 0$. The tangent direction of $R$ at $q$, corresponds to the double plane which is the tangent cone of $S$ at $p$.

    Assume for contradiction that $S$ is a symmetroid. \cref{lem:elliptic-rank2} and \cref{thm:ramification} imply that $R$ splits into two cubics, $R_1$ and $R_2$. This leaves two possibilities for consecutive triple points. One possibility is that $R_2$ does not pass through $q$, and $R_1$ is a triple line corresponding to the tangent direction at $q$. The second possibility is that $R_1$ breaks up into a line $L$ and a conic $C\mkern-1mu$, such that $L$, $C$ and $R_2$ all have the same tangent direction at $q$. In either case, the line corresponding to the tangent direction at $q$, is a component in $R$.

    Writing the equation for $S$ as $x_1^2x_0^2 + F_3x_0 + F_4$, as in \eqref{eq:quartic-double-point}, then the equation for $R$ becomes $F_3^2 - 4x_1^2F_4$. Since the line $\V(x_1)$ is a component of $R$, then $x_1$ is a factor in $F_3$. By \eqref{eq:type1}, we see that $S$ has a tacnode at $p$. This contradicts the assumption that $S$ is of type $S^{(3)}_4\mkern-4mu$.
\end{proof}

\section{Proof of \texorpdfstring{\cref{thm:main}}{Theorem 1.1}}

\noindent
By Noether's classification in \cite{Noe89}, we know that the only rational quartic surfaces are those with either a double curve, a triple point or an elliptic double point of type $S^{(1)}_4\mkern-4mu$, $S^{(2)}_4$ or $S^{(3)}_4\mkern-4mu$. Since we are dealing with irreducible quartic surfaces, we need only consider double curves of degree up to $3$. \cref{thm:double-conics-have-rank2,prop:no-twisted-cubic,prop:no-type2,prop:no-type3} show that there are no other possible rational quartic symmetroids than the ones listed in the theorem.

The claims about the general number of additional singularities are covered by \hyperref[prop:general-rank2-line]{Propositions} \ref{prop:general-rank2-line}, \ref{prop:general-rank3-line}, \ref{prop:general-conic}, \ref{prop:general-triple-point} and \ref{prop:general-tacnode}. The existence of these types of symmetroids is demonstrated in \hyperref[ex:counterexample-converse]{Examples} \ref{ex:counterexample-converse}, \ref{ex:tacnode}, \ref{ex:rank3-line}, \ref{ex:conic} and \ref{ex:triple-point}.

The dimensions of the different families are calculated in \hyperref[prop:dimension-rank2-line]{Propositions} \ref{prop:dimension-rank2-line}, \ref{thm:dimension-double-line}, \ref{prop:dimension-conic}, \ref{prop:dimension-triple} and \ref{prop:dimension-tacnode}.
\qed

\section{Quartic Symmetroids with Two Double Intersecting Lines}
\label{sec:two-lines}

\noindent
The case of a surface $S \coloneqq \V(F)$ with two double lines, $L_1$ and $L_2$, is an interesting example, so we treat it with special care. Because of the following proposition, we are only interested in $L_1 \cap L_2 \neq \varnothing$:

\begin{proposition}
    \label{prop:skew-lines}
    Let $S \subset \PP^3$ be a quartic surface with two double, skew lines $L_1$ and $L_2$. Then $S$ is not rational.
\end{proposition}

\begin{proof}
    Let $p \in \PP^3$ be a point outside of $L_1$ and $L_2$. Let $H$ be a plane that does not contain $p$. Let $L'_i$ be the projection of $L_i$ from $p$ to $H$. The lines $L'_1$ and $L'_2$ meet in a point, which corresponds to a line $L$ through $p$ that intersects both $L_1$ and $L_2$. If $p \in S$, then B{\'e}zout's theorem implies that $L \subset S$, since $L_1$ and $L_2$ are skew and double. It follows that $S$ is a scroll.
    
    Let $H$ be a plane that does not contain any of the lines in $S$. Then $H$ intersects each line in $S$ in a point. The hyperplane section $C \coloneqq H \cap S$ is a plane quartic curve with two double points. The curve $C$ has genus $1$, so $S$ has genus $1$. This proves the claim.
\end{proof}

\begin{remark}
    \label{rmk:two-double-lines}    
    For our purposes, we may therefore assume that $L_1 \coloneqq \V(x_1, x_2)$ and $L_2 \coloneqq \V(x_1, x_3)$. Since $S$ is singular along these, the terms in $F$ have either $x_1^2$, $x_1x_2x_3$ or $x_2^2x_3^2$ as a factor. It follows that $F$ satisfies the equation \eqref{eq:type1} of a tacnodal surface. However, a tacnode is defined as an isolated singularity. We may regard $S$ as a degeneration of tacnodal surfaces. In addition, $S$ is a degeneration of surfaces with one double line and a degeneration of surfaces with a double smooth conic section.
\end{remark}

\noindent
Next, we prove similar results to \cref{thm:main} for the different possible ranks:

\begin{proposition}
    Let $S \subset \PP^3$ be a generic, irreducible quartic symmetroid that is singular along two intersecting lines, $L_1$ and $L_2$, of \points{3}. Then $S$ has two isolated \nodes{2}.
\end{proposition}

\begin{proof}
    By \cref{lem:singularity-is-a-base-point,lem:curve-in-base-locus}, the quadrics along each line $L_i$ have a common singular point. \cref{lem:rank3-pencil} states that the \locus{2} along each line has length~$3$.

    As noted in \cref{rmk:two-double-lines}, the equation for $S$ satisfies \eqref{eq:type1}. It follows that all the arguments in the proof of \cref{prop:general-tacnode} hold here as well, showing that the \locus{2} is of length~$6$ outside of $p \coloneqq L_1 \cap L_2$. Moreover, $p$ is counted with multiplicity~$2$ in each of the \loci{2} of the pencils defined by the $L_i$. Hence each $L_i$ contains only one additional \point{2}~$p_i$.
    
    The proof of \cref{prop:general-rank3-line} shows that $p_i$ is counted with multiplicity~$2$ in the \locus{2} of $S$. Thus the \locus{2} has length~$2$ outside of $L_1$ and $L_2$, which proves the claim.
\end{proof}

\begin{proposition}
    Let $S \subset \PP^3$ be a generic, irreducible quartic symmetroid with two intersecting lines, $L_1$ and $L_2$, of \points{2}. Then $S$ has four additional, isolated \points{2}.
\end{proposition}

\begin{proof}
    Let $P_i$ be the pencil of quadrics that corresponds to the line $L_i$. By \cref{lem:rank2-pencil}, the base locus of $P_i$ consists of a plane~$H_i$ and a line~$l_i \not \subset H_i$. The associated quadric at the point~$L_1 \cap L_2$ is $H_1 \cup H_2$. It follows that $l_1 \subset H_2$ and $l_2 \subset H_1$. Thus the base locus of the net~$N$ spanned by $P_1$ and $P_2$, consists of the three lines $l_1$, $l_2$ and $H_1 \cap H_2$. The web $W(S)$ is the net~$N$ extended with another quadric~$Q \notin N\mkern-2mu$. Then $Q$ intersects the lines $l_1$, $l_2$ and $H_1 \cap H_2$ in two points each. Hence the base locus of $W(S)$ consists of six points. The symmetroid~$S$ is therefore uniquely determined by the base locus of $W(S)$.

    Choose coordinates such that $H_1 \coloneqq \V(x_2)$, $H_2 \coloneqq \V(x_3)$, $l_1 \coloneqq \V(x_0, x_3)$ and $l_2 \coloneqq \V(x_1, x_2)$. We can assume that the six base points are
    \begin{align*}
    [0 : 1 : 1 : 0], [0 :  1 : -1 :  0] &\in l_1,
    \\
    [1 : 0 : 0 : 1], [1 :  0 :  0 : -1] &\in l_2,
    \\
    [1 : 1 : 0 : 0], [1 : -1 :  0 :  0] &\in H_1 \cap H_2. 
    \end{align*}
    The quadrics passing through these base points are parametrised by the matrix
    \begin{align*}
        M
        \coloneqq
        \begin{bmatrix}
            x_{00} & \phantom{-}0      & x_{02} & \phantom{-}0      \\
              0    &           -x_{00} &   0    & \phantom{-}x_{13} \\
            x_{02} & \phantom{-}0      & x_{00} & \phantom{-}x_{23} \\
              0    & \phantom{-}x_{13} & x_{23} &           -x_{00}
        \end{bmatrix}
        \mkern-7mu.
    \end{align*}
    The symmetroid defined by $M$ has four isolated \points{2} in addition to the lines $L_1$ and $L_2$.
\end{proof}

\begin{proposition}
    \label{prop:one-of-each}
    Let $S \subset \PP^3$ be a generic, irreducible quartic symmetroid that is singular along two intersecting lines, $L_1$ and $L_2$. Suppose that the points along $L_1$ are generically \points{3}, and that $L_2$ consists of \points{2}. Then $S$ has two isolated \nodes{2}.
\end{proposition}

\begin{proof}
    By \cref{prop:general-rank2-line}, the \locus{2} of $S$ has length~$6$ outside of $L_2$. As in the proof of \cref{prop:general-rank3-line}, the two \points{2} in $L_1 \setminus L_2$ are counted with multiplicity~$2$ each in the \locus{2}. Hence there are two \points{2} outside of $L_1$ and $L_2$.
\end{proof}

\noindent
In the case of \cref{prop:one-of-each}, the rank along the lines depends on the matrix representation~\eqref{eq:matrix-representation}:

\begin{example}
    \label{ex:2-representations}
    The symmetroid~$S$ defined by the matrix
    \begin{align*}
        A_1
        \coloneqq
        \begin{bmatrix}
              0  &     x_0     & \phantom{-}{4x_1} & \phantom{-}2x_2   \\
             x_0 &    4x_3     & 2x_1 - 2x_3       & \phantom{-}0      \\
            4x_1 & 2x_1 - 2x_3 & -4x_1             & \phantom{2}{-x_2} \\
            2x_2 &      0      & \phantom{4}{-x_2} & \phantom{-2}x_3
       \end{bmatrix}        
    \end{align*}
    is singular along two lines, $L_1$ and $L_2$, and it has four isolated nodes, $p_1$, $p_2$, $p_3$ and~$p_4$. Only $L_1$, $p_1$, $p_2$ and two points on $L_2 \setminus L_1$ are contained in the \locus{2} of~$S$.

    The matrix
    \begin{align*}
        A_2
        \coloneqq
        \begin{bmatrix}
                0      & \phantom{-4}x_0 - 8x_3
                       & 4x_3                   & \phantom{-}2x_2 \\
            x_0 - 8x_3 & \phantom{-}4x_1 + 8x_3
                       & -2x_1 - 6x_3           & -2x_2           \\
              4x_3     & -2x_1 - 6x_3
                       & 4x_3                   & \phantom{-2}x_2 \\
              2x_2     & -2x_2
                       & \phantom{4}x_2         & \phantom{2}{-x_1}
        \end{bmatrix}
    \end{align*}
    has the same determinant as $A_1$. The matrices are not conjugates of each other, which can be verified by evaluating $A_1$ and $A_2$ at a point and check that they have different eigenvalues. With the representation given by $A_2$, the \locus{2} of~$S$ equals $L_2$, $p_1$, $p_3$ and two points on $L_1 \setminus L_2$.
\end{example}

\section{Examples}
\label{sec:examples}

\noindent We end with a few more examples, including demonstrations showing that Pl{\"u}cker's surface and the Steiner surface are symmetroids.

\begin{example}
    \label{ex:rank3-line}
    The matrix
    \begin{align*}
    \setlength{\arraycolsep}{6pt}
    \begin{bmatrix}
               0        & x_0 + x_1 & x_0 + x_1 & x_0 + x_2 + x_3 \\
           x_0 + x_1    &    x_3    &     0     &       x_0       \\
           x_0 + x_1    &     0     &    x_1    &       x_2       \\
        x_0 + x_2 + x_3 &    x_0    &    x_2    &       x_2
    \end{bmatrix}
    \end{align*}
    defines a generic quartic symmetroid with a double line of \points{3}. The symmetroid has three \points{2} on the double line and four outside of the line.
\end{example}

\begin{example}[Pl{\"u}cker's surface]
    \label{ex:plucker}
    The maximal number of isolated nodes on a quartic surface with a double line is eight. The surface $S$ satisfying that description is known as \emph{Pl{\"u}cker's surface} in the classical literature \cite[Article~83]{Jes16}. It is represented as a symmetroid by the matrix
    \begin{align*}
    \setlength{\arraycolsep}{6pt}
    \begin{bmatrix}
               0        & x_0 - x_1 + x_2 & x_0 - x_1 + x_3 & x_0 \\
        x_0 - x_1 + x_2 &        0        &       x_3       & x_1 \\
        x_0 - x_1 + x_3 &       x_3       &        0        & x_2 \\
              x_0       &       x_1       &       x_2       & 0
    \end{bmatrix}
    \mkern-7mu.
    \end{align*}
    The double line and six of the nodes are contained in the \locus{2}. 

    In \cite[Remark~5.4]{Ott+14}, Kummer surfaces are given the following interpretation: They are the nodal symmetroids where the associated web of quadrics contains a net that defines a twisted cubic curve. Hence the Kummer surfaces have six general base points. The base locus of $W(S)$ consists of six points, four of which are coplanar. Pl{\"u}cker's surface is therefore a degeneration of Kummer surfaces. 
\end{example}

\begin{example}
    \label{ex:conic}
    The matrix
    \begin{align*}
    \setlength{\arraycolsep}{6pt}
    \begin{bmatrix}
            0     &       x_0       &       x_1       & x_0 + x_2 \\
           x_0    &        0        & x_1 + x_2 + x_3 &    x_3    \\
           x_1    & x_1 + x_2 + x_3 &        0        &    x_0    \\
        x_0 + x_2 &       x_3       &       x_0       &     0
    \end{bmatrix}
    \end{align*}
    defines a symmetroid that is double along a smooth conic section and has four isolated \points{2}.
\end{example}

\begin{example}
    \label{ex:triple-point}
    The symmetroid defined by
    \begin{align*}
    \setlength{\arraycolsep}{6pt}
    \begin{bmatrix}
        x_0 & x_1 & x_2 &  0  \\
        x_1 & x_3 &  0  & x_2 \\
        x_2 &  0  & x_3 & x_1 \\
         0  & x_2 & x_1 & x_3
    \end{bmatrix}
    \end{align*}
    is an example of a generic symmetroid with a triple point. It has in addition six \points{2}.
\end{example}

\begin{example}
    \label{ex:tacnode}
    Let $S \subset \PP^3$ be a quartic symmetroid with a tacnode at the point $p \coloneqq [1 : 0 : 0 : 0]$. By \cref{lem:elliptic-rank2}, we may assume that the matrix defining $S$ is on the form
    \begin{align*}
        A
        \coloneqq
        \begin{bmatrix}
            a_0x_0 + l_{00} &      l_{01}     & l_{02} & l_{03} \\
            l_{01}     & a_1x_0 +  l_{11} & l_{12} & l_{13} \\
            l_{02}     &      l_{12}      & l_{22} & l_{23} \\
            l_{03}     &      l_{13}      & l_{23} & l_{33}
        \end{bmatrix}
        \mkern-7mu,
    \end{align*}
    where $a_0, a_1 \in \C \setminus \{0\}$ and the $l_{ij}$ are linear forms in $\C[x_1, x_2, x_3]$. The tangent cone of $S$ at $p$ is given by
    \begin{align}
        \label{eq:tangent-cone}
        \begin{vmatrix}
            l_{22} & l_{23} \\
            l_{23} & l_{33}
        \end{vmatrix}
        =
        l_{22}l_{33} - l_{23}^2.
    \end{align}
    Since $p$ is an elliptic double point, \eqref{eq:tangent-cone} is a square. If $l_{22}$, $l_{23}$ and $l_{33}$ are scalar multiples of each other, not all zero, then $\det(A)$ is on the form \eqref{eq:type1}. In general, $A$ does not define a tacnode for other ways of realising \eqref{eq:tangent-cone} as a square.

    Concretely, the matrix
    \begin{align*}
        \begin{bmatrix}
            x_0 + x_1 & \phantom{-}x_2 & x_3 & x_1 \\
            x_2       &     -x_0 + x_1 & x_2 & x_2 \\
            x_3       & \phantom{-}x_2 & x_1 &  0  \\
            x_1       & \phantom{-}x_2 &  0  & x_1
        \end{bmatrix}
    \end{align*}
    defines a symmetroid with a tacnode and six additional nodes.
\end{example}

\begin{example}[Steiner surface]
    The Veronese surface $V \subset \PP^5$ can be identified with the \locus{1} of the symmetroid $S$ defined by the matrix
    \begin{align*}
    \setlength{\arraycolsep}{6pt}
    \begin{bmatrix}
        x_{00} & x_{01} & x_{02} \\
        x_{01} & x_{11} & x_{12} \\
        x_{02} & x_{12} & x_{22}
    \end{bmatrix}
    \mkern-7mu.
    \end{align*}
    Moreover, $S$ is the secant variety of $V$ \cite[Section~2.1.1]{Dol12}.
    
    The general projection $R$ of $V$ to $\PP^3$ is known as the \emph{Roman surface} or the \emph{Steiner surface}. It is an irreducible surface with three concurrent double lines, $L_1$, $L_2$ and $L_3$. They meet in a triple point $p$. After a suitable choice of coordinates, we may assume that $R$ is given by \cite[Equation~(2.1)]{Dol12}, which is
    \begin{align*}
        x_0x_1x_2x_3 + x_0^2x_1^2 + x_0^2x_2^2 + x_1^2x_2^2 = 0.
    \end{align*}
    The determinant of the matrix
    \begin{align*}
    \setlength{\arraycolsep}{6pt}
    A
    \coloneqq
    \begin{bmatrix}
    0   &     x_0     &   x_1 &             x_0           \\
    x_0 &      0      & -2x_2 &         4x_0 + 2x_2       \\
    x_1 &   -2x_2     &    0  &            2x_1           \\
    x_0 & 4x_0 + 2x_2 &  2x_1 &  6x_0 + 2x_1 + 2x_2 - x_3
    \end{bmatrix}
    \end{align*}
    is
    \begin{align*}
    \det(A) = 4 \big( x_0x_1x_2x_3 + x_0^2x_1^2 + x_0^2x_2^2 + x_1^2x_2^2 \big)
    \mkern-3mu,
    \end{align*}
    so the Steiner surface is a symmetroid. The triple point $p$ is a \point{1}. There is also a single \point{2}, with multiplicity $2$, on each line $L_i$.

    The base locus of $W(R)$ is a scheme of length $6$. It consists of three points and a direction through each. As with Pl{\"u}cker's surface in \cref{ex:plucker}, the Steiner surface is a degeneration of Kummer surfaces.
\end{example}

\begin{figure}[htbp]
    \centering
    \includegraphics[height = 0.28\textheight]{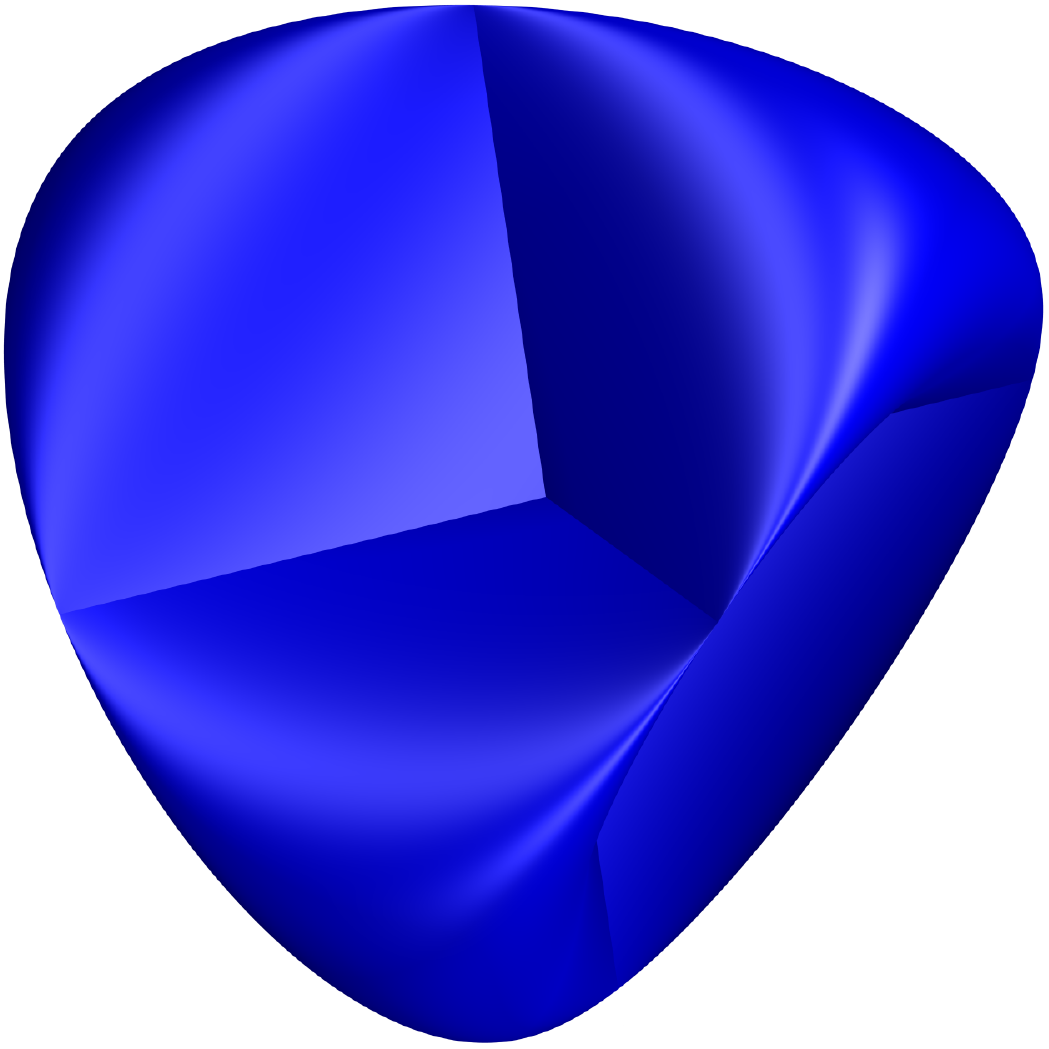}
    \caption{The Steiner surface.}
\end{figure}

\begin{example}
    The symmetroid $S$ defined by
    \begin{align*}
    \setlength{\arraycolsep}{6pt}
    \begin{bmatrix}
    0  &     x_0    &     x_1    & x_0 \\
    x_0 &      0     & 2x_2 - x_3 & x_2 \\ 
    x_1 & 2x_2 - x_3 &      0     & x_0 \\
    x_0 &     x_2    &     x_0    & x_3
    \end{bmatrix}
    \end{align*}
    is double along a smooth conic section $C$ and a line $L$. The \locus{2} consists of~$C$ and a point $p \in L \setminus C\mkern-1mu$. The base locus of $W(S)$ is a scheme of length $4$ with support in three points.
\end{example}

\begin{acknowledgements}
    I would like to thank Kristian Ranestad for insightful discussions and for supervising the master's thesis of which this paper is both an abridgement and an extension.
\end{acknowledgements}

{
    \pretolerance = 1275
    \printbibliography
}

\end{document}